\documentclass[a4paper,fleqn]{cas-sc}
\usepackage{blindtext}
\usepackage[utf8]{inputenc}

\usepackage[numbers]{natbib}

\usepackage{graphicx}
\usepackage{mathtools}
\usepackage{amssymb}
\usepackage{amsfonts}
\usepackage{amsthm}
\usepackage{amsmath}
\usepackage{xcolor}
\usepackage{float}
\usepackage{caption}
\usepackage{subcaption}
\usepackage{comment}
\usepackage{enumitem}
\usepackage{mathrsfs}
\usepackage{placeins}

\usepackage{tikz-cd}

\usepackage{setspace}

\usepackage{hyperref}
\hypersetup{
	colorlinks=true,
	linkcolor=blue,
	filecolor=magenta,      
	urlcolor=cyan,
	pdfpagemode=FullScreen,
}

\usepackage{cases}
\usepackage{orcidlink}

\def\tsc#1{\csdef{#1}{\textsc{\lowercase{#1}}\xspace}}
\tsc{WGM}
\tsc{QE}

\newcounter{assumption}

\newenvironment{assumption}[1]{%
	\def\theassumption{#1}%
	\refstepcounter{assumption}%
	\begin{quote}%
		\textbf{(\theassumption) }\ignorespaces
	}{%
	\end{quote}%
}

\newcommand{\assref}[1]{\hyperref[#1]{(\ref*{#1})}}

\DeclareMathOperator{\Span}{Span}

\newcommand{\Bd}{\mathcal B_d}
\newcommand{\Hd}{\mathcal H_d}
\newcommand{\Int}{\operatorname{Int}}
\newcommand{\parv}{\operatorname{par}}
\newcommand{\ctr}{\operatorname{ctr}}
\newcommand{\lca}{\operatorname{lca}}
\newcommand{\MovU}[1]{\overline{\mathfrak U}^{\,#1}}
\newcommand{\MU}[1]{\mathfrak U^{\,#1}}
\newcommand{\ovU}[1]{\overline{U}^{\,#1}}
\newcommand{\R}{\mathbb R}

\newcommand{\Iint}{I_\sigma}

\newcommand{\cl}[1]{\mathcal{#1}}

\newtheorem{proposition}{Proposition}
\newtheorem{theorem}{Theorem}
\newtheorem{corollary}{Corollary}
\newtheorem{lemma}{Lemma}
\newdefinition{definition}{Definition}
\newdefinition{remark}{Remark}

\def\R{\mathbb{R}}

\def\P{\mathbb{P}}

\allowdisplaybreaks

\begin{document}
	\let\WriteBookmarks\relax
	\def\floatpagepagefraction{1}
	\def\textpagefraction{.001}
	
	\shorttitle{Polyhedral inference for hierarchical copulas}    
	
	\shortauthors{L. Marconi}  
	
	\title [mode = title]{A polyhedral framework for hierarchical Archimedean copulas: estimation, hierarchy recovery and singular asymptotics}  
	
	
	
	%
	
	\author[1]{Leonardo Marconi}[orcid=0009-0000-1597-7875]
	
	\cormark[1]
	
	
	\ead{leonardo.marconi5@unibo.it}
	
	
	\credit{Conceptualization, Methodology, Formal analysis,
		Investigation, Writing -- original draft, Writing -- review \& editing}
	
	\affiliation[1]{organization={Alma Mater Studiorum University of Bologna},
		addressline={Via delle Belle Arti 41}, 
		city={Bologna},
		postcode={40126}, 
		country={Italy}}
	
%
%
%
%
	
	
%
%
%
%
%
	
	\cortext[1]{Corresponding author}
	
	
	
	\begin{abstract}
		Hierarchical Archimedean copulas (HACs) admit redundant tree representations: when adjacent internal nodes have equal parameters, contracting the connecting edge leaves the copula unchanged. This creates a structural overparametrization. For homogeneous one-parameter families, we remove this redundancy and identify the resulting reduced model class with the ultrametric fan. We treat the Clayton, Frank, Ali-Mikhail-Haq, Gumbel and Joe families.
		
		Using the negative logit of normalized Kendall's tau as a common node-height coordinate, reduced hierarchies form polyhedral strata, binary hierarchies index maximal cones, and multifurcating hierarchies lie on common contraction faces. The induced cophenetic metric generates the same topology as uniform convergence of copulas. We also show that informative pairwise dependence summaries act componentwise on cophenetic coordinates and hence provide global coordinates for the reduced model.
		
		This representation yields a global estimator obtained by inverting estimated pairwise summaries and projecting onto the fan, without preselecting a hierarchy. We prove consistency and convergence-rate transfer. At binary truths, the unthresholded estimator recovers the hierarchy with probability tending to one; at multifurcating truths, it asymptotically selects a binary refinement, while vanishing thresholding consistently recovers the reduced hierarchy. Under a joint central limit theorem for the pairwise summaries, the projected estimator has a root-$n$ limit given by Euclidean projection of a Gaussian vector onto the tangent fan. The limit is Gaussian at binary hierarchies and generally non-Gaussian at multifurcating ones. For empirical Kendall's tau, we verify the required regularity conditions and obtain the corresponding singular limit.
\end{abstract}

	
	\begin{highlights}
		\item Reduced hierarchical Archimedean copulas form a finite polyhedral fan.
		\item Pairwise dependence summaries give global coordinates on the reduced model.
		\item Global projection estimation is consistent without preselecting a hierarchy.
		\item Multifurcating truths are over-resolved without vanishing thresholding.
		\item Tangent-fan projection yields generally non-Gaussian root-\(n\) limits.
	\end{highlights}

	
	\maketitle
	
	
	
	
	

	%
	%
	
	
	

	\section{Introduction} \label{sec:introduction}
	
	Archimedean copulas provide a parsimonious class of multivariate dependence models, but their exchangeability can be restrictive when different groups of variables exhibit different levels of dependence. Hierarchical, or nested, Archimedean copulas (HACs) address this limitation by recursively composing Archimedean copulas along a rooted tree. The internal vertices carry dependence parameters, while the tree records how the variables are grouped. This construction retains much of the analytical tractability of ordinary Archimedean copulas while permitting non-exchangeable dependence structures; see, among others, \citep{HofertMaechler2011,Joe1997,McNeil2008}.
	
	Statistical inference for these models involves both continuous and discrete unknowns: one must estimate the parameters attached to the internal vertices and recover the underlying hierarchy. Several methods have been proposed for this purpose. In the homogeneous case, pairwise margins and rank-based dependence
	measures, particularly Kendall's tau, lead naturally to reconstruction
	procedures \citep{SegersUyttendaele2014}. Hierarchical aggregation
	methods for structure and parameter estimation were developed, among
	others, by \citet{OkhrinOkhrinSchmid2013}. Other approaches fit an initially binary hierarchy and subsequently collapse selected forks, use penalization or goodness-of-fit comparisons, or estimate the hierarchy, the generator families and their parameters simultaneously \citep{GoreckiHofertHolena2016,GoreckiHofertHolena2017,MatsypuraNeoProkhorov2016,Uyttendaele2014,Uyttendaele2018}. In particular, \citet{OkhrinRistig2024} propose a penalized likelihood procedure that simultaneously estimates parameters and aggregates the HAC structure by shrinking differences between adjacent parameters to zero. Likelihood-based inference for Archimedean and hierarchical constructions provides another important perspective \citep{HofertMaechlerMcNeil2012,HofertPham2013}.
	
	A fundamental difficulty is that an arbitrary decorated hierarchy is not an identifiable parametrization of the corresponding copula. If an internal vertex and one of its internal children carry the same Archimedean parameter, associativity implies that contracting the edge between them leaves the represented copula unchanged. A binary hierarchy with equal parameters along some edges therefore represents the same probability model as a lower-dimensional multifurcating hierarchy. A probability model with multifurcations may lie simultaneously in the closures of several higher-dimensional tree parametrizations. Ordinary smooth-model asymptotics therefore need not describe estimation when the hierarchy itself is unknown.
	
	The work most directly related to this singular inferential problem is \citet{PerreaultTangPanReid2025}. They study potentially overparametrized hierarchical Archimedean copulas and derive stochastic representations for constrained maximum likelihood estimators, leading in particular to likelihood-ratio tests for structural hypotheses associated with collapsed nodes. Their analysis starts from a prescribed overparametrized hierarchy and studies the local effect of its parameter constraints. The question considered here is complementary. We first identify all decorated hierarchies that represent the same copula and work directly on the resulting space of reduced models. The hierarchy is then allowed to vary over all refinements of the true reduced hierarchy. Consequently, at a multifurcating copula the local feasible set is generally a non-convex union of refinement cones rather than the constraint cone associated with one prescribed tree. This distinction is essential for describing not only continuous parameter fluctuations but also the first-order hierarchy selected by the estimator.
	
	This viewpoint suggests that reduced hierarchical copulas should not be treated as a disjoint collection of tree-specific parameter spaces. Binary models possess boundary faces on which adjacent parameters become equal, and these faces represent exactly the multifurcating models obtained by contracting the corresponding edges. Tree topology and continuous node parameters therefore form a single stratified model space. The central purpose of this paper is to identify this space explicitly and to exploit its structural geometry for statistical inference.
	
	The relevant geometric object is familiar from phylogenetics, combinatorics and tropical geometry. A rooted equidistant tree can be encoded by its cophenetic vector, whose coordinates record the heights of the lowest common ancestors (LCAs) of all pairs of leaves. Such vectors satisfy the ultrametric three-point condition and form the Bergman fan of the complete graph. Its maximal cones correspond to binary hierarchies, while their lower-dimensional faces correspond to non-binary hierarchies obtained by contracting zero-length internal edges \citep{Ardila2005,ArdilaKlivans2006}. Cophenetic-vector metrics for rooted trees have been studied in phylogenetics \citep{CardonaMirRosselloRotgerSanchez2013}, while Euclidean least-squares projection onto the space of equidistant trees has been investigated as a tree-reconstruction problem \citep{FaheyHostenKriegerTimpe2008}.
	
	In the present work, we realize one-parameter hierarchical Archimedean
	copulas as the Bergman fan. Admissible families include Clayton, Frank, Ali-Mikhail-Haq (AMH), Gumbel and Joe. We use the negative logit of normalized Kendall's
	tau as a common node-height coordinate.
	In these coordinates every reduced hierarchy
	determines the relative interior of a polyhedral cone, while setting selected
	parent--child increments equal to zero produces exactly the faces associated
	with edge contractions. Through the resulting bijection, the polyhedral
	stratification of the ultrametric fan is transported to the copula class:
	fixed reduced hierarchies form its strata, and multifurcating models are the
	common contraction faces of their refinements. We likewise transport the
	ambient Euclidean metric to obtain a cophenetic metric on the model space and
	show that it induces exactly the topology inherited from uniform convergence
	of copula functions. Thus the reduced copula class acquires a
	family-independent hierarchy-stratified metric geometry. On this
	common stratified space, statistical geometries such as the Fisher--Rao
	geometry may assign family-dependent lengths to the same contraction and
	refinement directions.
	
	The same representation clarifies the role of pairwise dependence
	information. It is known that the hierarchy of a hierarchical Archimedean
	copula can be recovered from its bivariate margins, and pairwise
	rank-based information has accordingly been used for structure and parameter
	estimation \citep{GoreckiHofertHolena2016,OkhrinOkhrinSchmidProperties2013,SegersUyttendaele2014}. In cophenetic coordinates, every bivariate copula
	functional acts componentwise through a scalar profile.
	Whenever this profile is injective, the full matrix of pairwise summaries therefore
	provides global coordinates on the reduced model space, recovering
	simultaneously the copula's hierarchy and parameters. We verify the required
	identifiability and regularity properties for Kendall's tau and Spearman's rho
	in the admissible families, and
	for the lower- or upper-tail dependence coefficients whenever they are informative.
	
	These coordinates lead naturally to a global estimator on the complete
	reduced model space. Starting from estimated pairwise summaries, we invert the
	scalar profile coordinatewise to obtain an unconstrained cophenetic vector and project it onto the Bergman fan. This projection is then bijectively mapped to the stratified HAC model
	space. The resulting estimator does not require a binary hierarchy to be
	specified in advance. We establish its consistency and transfer convergence
	rates from the initial pairwise estimator.
	
	A central consequence of the stratification is a sharp distinction between
	estimation of the copula and recovery of its reduced hierarchy. Binary
	hierarchies lie in the relative interiors of maximal cones and are locally
	stable, so consistent estimation recovers their hierarchy. Multifurcating hierarchies instead lie on common faces of
	several refinement cones. As a consequence,  without
	thresholding, the estimated hierarchy is a binary
	refinement of the true reduced hierarchy with probability tending to one.
	Therefore, exact unthresholded hierarchy recovery has a sharp dichotomy:
	its probability tends to one at a binary truth and to zero at a
	multifurcating truth. We show that the reduced hierarchy can nevertheless be
	consistently recovered by contracting sufficiently small estimated increments,
	and formulate the corresponding vanishing-threshold recovery theorem for
	general consistent estimators.
	
	We explain this phenomenon through a singular asymptotic theory for the
	projected estimator. Under a joint central limit theorem for the initial
	pairwise summaries, coordinatewise inversion produces a Gaussian perturbation
	in the ambient cophenetic space. At a multifurcating true copula every
	refinement of the reduced hierarchy is locally feasible, and we prove that
	the exact first-order constraint is a generally non-convex tangent fan.
	The root-\(n\) limit of the projected estimator is the nearest-point
	projection of the Gaussian perturbation onto this fan. At a binary hierarchy
	the tangent fan reduces to a linear space and the limit is Gaussian; at a
	multifurcating hierarchy it is generally non-Gaussian. Moreover, the limiting
	projection avoids the non-maximal skeleton almost surely, which is the local
	geometric mechanism behind the asymptotic selection of a binary refinement.
	The additional refinement increments are therefore of order \(n^{-1/2}\).
	Accordingly, any threshold satisfying
	\[
	\varepsilon_n\downarrow0,
	\qquad
	\sqrt n\,\varepsilon_n\to\infty
	\]
	removes these sampling-induced refinement edges while retaining the genuine
	positive increments, and hence consistently recovers the true reduced
	hierarchy.
	
	The Gaussian projection limit reflects the general tangent-cone
	principle for constrained estimation and singular likelihood
	asymptotics \citep{Drton2009,Geyer1994,Shapiro2000}.
	We establish it directly from the exact local polyhedral
	geometry of the model, revealing how Gaussian perturbations
	select among refinements of the true reduced hierarchy. Rather than a single constraint cone attached to
	a prescribed parametrization, the local feasible set consists of all
	refinement cones meeting at the true reduced model. The limiting projection
	therefore describes simultaneously the root-\(n\) fluctuations of the genuine
	continuous parameters and the random first-order resolution of structural
	uncertainty among competing refinements. This geometry is also related to
	nonclassical central limit phenomena on stratified and phylogenetic tree
	spaces \citep{BardenLeOwen2013,HotzEtAl2013}.
	
	The remainder of the paper is organized as follows. Section~\ref{sec:ultrametric-hierarchy-space} constructs the ultrametric fan and establishes its local tangent geometry. Section~\ref{sec:archimedean-realization} realizes the fan as the reduced homogeneous hierarchical Archimedean copula class. Section~\ref{sec:copula-geometry} studies the cophenetic metric and pairwise dependence coordinates. Section~\ref{sec:statistical-recovery} introduces the projection estimator and establishes consistency and hierarchy recovery. Section~\ref{sec:singular-clt} derives the tangent-fan asymptotic distribution. Section \ref{sec:kendall} verifies the assumptions and specializes all results to the empirical Kendall matrix.

	\section{The ultrametric hierarchy space}
	\label{sec:ultrametric-hierarchy-space}
	
	The geometry developed in this section is largely classical. Ultrametric vectors, their polyhedral fan structure, and their relation with rooted equidistant trees are well known in phylogenetics, combinatorics and tropical geometry; see, for instance, \citep{Ardila2005,ArdilaKlivans2006}.
	We recall here the features of this geometry that will be needed below, in a notation adapted to hierarchical copulas, and derive the specific local statements required for the statistical analysis. Accordingly, proofs of standard geometric facts are kept concise, while details are given when they are needed to establish a formulation used later in the paper.

	Fix \(d\ge2\) and write \([d]:=\{1,\ldots,d\}\). All trees below are
	rooted, non-planar, and have leaves bijectively labelled by \([d]\).
	Non-planar means that the ordering of the children of a vertex is not part of
	the structure.
	
	Let \(\Bd\) denote the finite set of full binary hierarchies on \([d]\),
	i.e. rooted leaf-labelled trees in which every internal vertex has exactly
	two children. Let \(\Hd\) denote the finite set of all rooted
	leaf-labelled hierarchies on \([d]\), where every internal vertex has at
	least two children. Thus \(\Bd\subset\Hd\). For \(\sigma\in\Hd\), write
	\(\Int(\sigma)\) for the set of internal vertices, including the root
	\(\rho\), and \(\Iint:=\Int(\sigma)\setminus \{\rho\}\) for the set of internal vertices minus the root. For every non-root vertex \(v\), including leaves, write
	\(\parv_\sigma(v)\) for its parent. When the hierarchy is fixed,
	we suppress the subscript and write simply \(\parv(v)\). Set \(K(\sigma):=|\Int(\sigma)|\). If \(\sigma\in\Bd\), then
	\(K(\sigma)=d-1\).
	
	For \(\sigma\in\Hd\), define the open and closed coordinate domains
	\[
	U^\sigma:=\mathbb R\times(0,\infty)^{K(\sigma)-1},
	\qquad
	\ovU{\sigma}:=\mathbb R\times[0,\infty)^{K(\sigma)-1}.
	\]
	We write \(x=(r,\eta)\), where \(r\) is the root height and
	\(\eta=(\eta_v)_{v\in\Iint}\) is the vector of
	parent-child increments. The definition of $U^\sigma$ excludes zero internal increments:
	if $(r,\eta)\in U^\sigma$, then every non-root internal vertex has
	strictly smaller height than its parent. In that case, we refer to the pair \((\sigma,x)\)
	as \emph{reduced}. Accordingly, when we refer to a hierarchy as reduced it is understood that all corresponding increments are positive. By contrast, $\overline U^{\,\sigma}$ allows zero
	increments, corresponding to internal edges whose endpoints have the
	same height. The internal heights are defined recursively by
	\begin{equation}
		\label{eq:h-to-r-delta-on-quotient}
		h_\rho(x):=r,
		\qquad
		h_v(x):=h_{\parv(v)}(x)-\eta_v,
		\qquad v\in\Iint.
	\end{equation}
	Thus the heights strictly decrease from parent to child on \(U^\sigma\),
	whereas \(\ovU{\sigma}\) also allows equal heights along some
	internal edges.
	
	Write \(N:=\binom{d}{2}\). We identify \(\R^N\) with the space
	\(\R^{\binom{[d]}2}\) of real vectors indexed by the unordered pairs of
	distinct leaves. We write such a vector as
	 \(
	 H=(H_{ij})_{1\le i<j\le d},
	 \)
	and, whenever both orders of the indices occur, adopt the symmetric
	convention \(H_{ji}:=H_{ij}\).
	
	For an internal vertex \(v \in \Int(\sigma)\), let \(A_v\subset[d]\) be the cluster of
	leaves below \(v\). In particular, \(A_\rho=[d]\).
	The internal clusters form a laminar family \(\mathscr A(\sigma)\): for any \(v,w\in\Int(\sigma)\), the sets \(A_v\) and \(A_w\) are either disjoint or one contains the other. Conversely, a laminar family of non-singleton subsets of \([d]\) containing \([d]\) determines a unique hierarchy in \(\Hd\), \cite[Theorem 3.5.2]{SempleSteel2003}. We identify \(v\) with \(A_v \in \mathscr A(\sigma)\). For \(A\in\mathscr A(\sigma)\setminus\{[d]\}\), we write
	\(
	\parv_\sigma(A):=A_{\parv_\sigma(v)}.
	\)
	Equivalently, \(\parv_\sigma(A)\) is the unique minimal internal cluster
	of \(\sigma\) that strictly contains \(A\). As above, the subscript is
	omitted when the hierarchy is clear.
	
	For every
	\(A\subset[d]\) with \(|A|\ge2\), define
	\(b_A\in\mathbb R^N\) by
	\(
	(b_A)_{ij}:=\mathbf 1_{\{\{i,j\}\subset A\}},
	\text{ for } 1\le i<j\le d,
	\)
	and let \(\mathbf 1\in\mathbb R^N\) denote the vector whose
	coordinates are all equal to one. Define the linear map
	\begin{equation}
		\label{eq:linear-cophenetic-map}
		L_\sigma: \R^{K(\sigma)}\to \R^N, \quad L_\sigma(r,\eta)
		:=r\mathbf 1-
		\sum_{v\in\Iint}
		\eta_v b_{A_v}.
	\end{equation}
	For \(1\le i<j\le d\), let \(\lca_\sigma(i,j)\) be the LCA of the leaves \(i\) and \(j\) in \(\sigma\), that is the node $v$ such that $A_v$ is the minimal cluster with \(\{i,j\}\subset A_v.\)
	The coordinates of \(L_\sigma(r,\eta)\) are the LCA heights. Indeed, for every \(1\le i<j\le d\),
	\begin{equation}
		\label{eq:LCA-coordinate}
		L_\sigma(r,\eta)_{ij}
		=
		r-\sum_{\substack{v\in I_\sigma \\
				\{i,j\}\subseteq A_v}}\eta_v
		=
		h_{\lca_\sigma(i,j)}(r,\eta).
	\end{equation}
	where \(v \prec w\) means that \(v\) lies in the path between the root and \(w\).

	We set
	\begin{equation}
		\label{eq:hierarchy-cones}
		\mathfrak U^\sigma
		:=L_\sigma(U^\sigma),
		\qquad
		\MovU{\sigma}
		:=L_\sigma(\ovU{\sigma}), \qquad V_\sigma
		:=\Span\bigl(\MovU{\sigma}\bigr).
	\end{equation}

We use the following standard terminology for convex cones. For a convex set
\(C\subset\mathbb R^N\), we denote by
\(
\operatorname{relint}(C)
\)
its relative interior, that is, its interior with respect to the topology of
its affine hull. For a convex cone \(C\), its
lineality space is
\(
\operatorname{lin}(C):=C\cap(-C),
\)
the largest linear subspace contained in \(C\). A pointed polyhedral cone is
called simplicial if it is generated by linearly independent extreme rays.

	\begin{proposition}
		\label{prop:linear-hierarchy-cell}
		For every \(\sigma\in\Hd\), the map \(L_\sigma: \R^{K(\sigma)}\to \R^N\) is injective and
		\begin{equation}
			\label{eq:cone-representation}
		\MovU{\sigma}
		=\mathbb R\mathbf 1+
		\operatorname{cone}\{-b_{A_v}:v\in\Iint\}=\mathbb R\mathbf 1+
		\biggl\{-\sum_{ v\in\Iint}\alpha_v b_{A_v}: \alpha_v \ge 0\biggr\}.
		\end{equation}
		Consequently, \(\MovU{\sigma}\) is a closed convex
		polyhedral cone of dimension \(K(\sigma)\), with lineality space
		\(\mathbb R\mathbf 1\) and relative interior \(\mathfrak U^\sigma\). Moreover,
		\(
		 V_\sigma
		=
		L_\sigma(\R^{K(\sigma)}).
		\)
	\end{proposition}
	
	\begin{proof}
		To prove injectivity, for every
		internal vertex \(v\), choose two leaves \(i_v,j_v\) lying in two distinct
		child subtrees of \(v\). Then \(\lca_\sigma(i_v,j_v)=v\), so the coordinate
		\(L_\sigma(r,\eta)_{i_vj_v}\) recovers \(h_v\). In particular,
		\(r=h_\rho\), and every non-root increment is recovered through
		\(\eta_v=h_{\parv(v)}-h_v\). Hence \(L_\sigma\) is injective.
		
		Therefore \(L_\sigma\) is a linear isomorphism onto its image and, consequently, it carries
		\(\ovU{\sigma}\) onto a closed polyhedral cone
		of dimension \(K(\sigma)\), and it carries
		\(U^\sigma\) onto its relative interior. Equation \eqref{eq:cone-representation} follows directly from
		\eqref{eq:linear-cophenetic-map}.

		The lineality space of \(\overline U^\sigma\) is
		\(
		\mathbb R\times\{0\}.
		\)
		Because \(L_\sigma\) is injective, one has
		\(
		\operatorname{lin}(\overline{\mathfrak U}_\sigma)
		=
		L_\sigma(\mathbb R\times\{0\})
		=
		\mathbb R\mathbf 1.
		\)		
		Finally, \eqref{eq:linear-cophenetic-map} and \eqref{eq:cone-representation} give
		\(
		 V_\sigma
		=
		\Span(
		\mathbf 1,\{b_{A_v}:v\in I_\sigma\})
		=
		L_\sigma(\R^{K(\sigma)}).
		\)

%
	\end{proof}
	
	For \(\gamma,\sigma\in\Hd\), write \(\gamma\preceq\sigma\) if
	\(\gamma\) is obtained from \(\sigma\) by contracting a possibly empty
	collection of non-root internal edges, and write
	\(\gamma\prec\sigma\) if \(\gamma\preceq\sigma\) and
	\(\gamma\neq\sigma\).
	
	If \(B\subseteq\Iint\), define
	\(\sigma/B\) to be the hierarchy whose internal cluster family is
	\[
	\mathscr A(\sigma/B)
	=
	\mathscr A(\sigma)
	\setminus
	\{A_v:v\in B\}=\{A_v : v \in \Int(\sigma)\setminus B\}.
	\]
	Equivalently, \(\sigma/B\) is obtained by contracting the edges joining
	the vertices in \(B\) to their parents.
	
	For \(\sigma,\tau\in\Hd\), denote by \(\sigma\wedge\tau\) the	hierarchy whose internal cluster family is
	\(
	\mathscr A(\sigma\wedge\tau)
	:=
	\mathscr A(\sigma)\cap\mathscr A(\tau).
	\)
	Since the intersection of two laminar families is again laminar and
	contains \([d]\), this defines an element of \(\Hd\).
	
	We now define the hierarchy space directly in cophenetic coordinates. Let
	\begin{equation}
		\label{eq:ultrametric-fan-support}
		\mathfrak U_d
		:=\left\{
		H\in\mathbb R^{N}:
		\max\{H_{ij},H_{ik},H_{jk}\}
		\text{ is attained at least twice for all distinct }i,j,k
		\right\}.
	\end{equation}
	The entries of \(H\) are not required to be non-negative. The translation
	\(H\mapsto H+c\mathbf 1\) changes the root height while leaving all
	internal increments unchanged.
	
	\begin{theorem}
		\label{thm:ultrametric-polyhedral-characterization}
		Every \(H\in\mathfrak U_d\) determines a unique hierarchy
		\(\sigma\in\mathcal H_d\) and a unique coordinate vector
		\((r,\eta)\in U^\sigma\) such that
		\(H=L_\sigma(r,\eta)\).
		Consequently,
		\begin{equation}
			\label{eq:ultrametric-stratification}
			\mathfrak U_d
			=\bigsqcup_{\sigma\in\Hd}\mathfrak U^\sigma.
		\end{equation}
		Moreover, the following statements hold.
		\begin{enumerate}[label=\textup{(\roman*)}]
			\item \(
			\MovU{\sigma}
			=\bigsqcup_{\gamma\preceq\sigma}\mathfrak U^\gamma,
			\) for every \(\sigma\in\Hd\).

			\item Set \(\sigma \in \cl H_d.\) If \(B\subset\Iint\), then the face of \(\MovU\sigma\) on
			which \(\eta_v=0\) for \(v\in B\) is
			\(
			\MovU{\sigma/B}.
			\)
			Moreover, every face of
			\(\MovU{\sigma}\) is obtained in this way.
			
			\item 
			\(
			\MovU{\sigma}
			\cap
			\overline{\mathfrak U}^{\,\tau}
			=
			\overline{\mathfrak U}^{\,\sigma\wedge\tau},
			\) for \(\sigma,\tau\in\Hd\).
			
			\item The collection
			\(
			\mathcal F_d
			:=\{\MovU{\sigma}:\sigma\in\Hd\}
			\)
			is a finite polyhedral fan with common lineality space
			\(\mathbb R\mathbf 1\) and support \(\mathfrak U_d\). Its maximal cones
			are precisely those indexed by binary hierarchies and have dimension
			\(d-1\). In particular, \(\mathfrak U_d\) is closed in
			\(\mathbb R^{N}\).
		\end{enumerate}
	\end{theorem}
	
	\begin{proof}
		Let \(H\in\mathfrak U_d\). For \(t\in\mathbb R\), define
		\[
		i\sim_t j
		\quad\Longleftrightarrow\quad
		i=j\ \text{ or }\ H_{ij}\le t.
		\]
		The relation is transitive: if \(H_{ij}\le t\) and \(H_{jk}\le t\)
		but \(H_{ik}>t\), then \(H_{ik}\) would be the unique maximum of
		\(\{H_{ij},H_{ik},H_{jk}\}\), contradicting
		\eqref{eq:ultrametric-fan-support}. Reflexivity holds by definition, while symmetry follows from the convention \(H_{ij}=H_{ji}\). Hence \(\sim_t\) is an equivalence
		relation, and its partitions are nested as \(t\) increases, meaning that, for $s \le t$, every equivalence class at level $s$ is contained in an equivalence class at level $t$.
		
		Let \(\mathscr A(H)\) be the collection of all equivalence classes of
		\(\sim_t\), as \(t\) varies, having cardinality at least two. Since its partitions are nested, this is a
		finite laminar family containing \([d]\), and therefore determines a
		unique rooted non-planar hierarchy \(\sigma(H)\), \cite[Theorem 3.5.2]{SempleSteel2003}. Every
		\(A\in\mathscr A(H)\) is uniquely of the form \(A=A_v\) for some
		\(v\in\Int(\sigma(H))\). For \(A\in\mathscr A(H)\), set
		\(
		h_A:=\max\{H_{ij}:i,j\in A,\ i<j\}.
		\)
		
		If \(A\subsetneq B\) are two such clusters, then \(h_A<h_B\). Indeed,
		choose \(a\in A\) and \(b\in B\setminus A\). Since \(A\) is already an
		equivalence class at level \(h_A\), one has \(H_{ab}>h_A\), while
		\(h_B\ge H_{ab}\).
		
		For \(i < j\), denote by \(A_{ij}\) the equivalence class at threshold level \(t=H_{ij}\). We have \(h_{A_{ij}}=H_{ij}\). Moreover, \(A_{ij}\) is the smallest class containing both \(i\) and \(j\). Indeed, if \(\{i,j\}\subseteq B\) then \(h_B \ge H_{ij}\) and thus \(B \supset A_{ij}\). Therefore, \(A_{ij}\) is identical to the cluster corresponding to their common ancestor, that is \(A_{ij}=A_{\lca_\sigma(i,j)}\). It follows that
		\(
		H_{ij}=h_{\lca_{\sigma(H)}(i,j)}.
		\)
		
		Set \(r:=h_{[d]}\) and
		\(\eta_A:=h_{\parv(A)}-h_A>0\) for every non-root internal cluster.
		Then \(H=L_{\sigma(H)}(r,\eta)\). The relations \(\sim_t\), and hence \(\mathscr{A}(H)\) and the resulting hierarchy \(\sigma(H)\), are uniquely determined by \(H\). Once \(\sigma(H)\) is known, Proposition~\ref{prop:linear-hierarchy-cell}
		recovers the coordinates uniquely. This gives
		\[\mathfrak U_d \subset \bigcup_{\sigma \in \cl H_d}\mathfrak U^\sigma.\]
		
		To prove the opposite inclusion, let \(H=L_\sigma(r,\eta)\) with \((r,\eta)\in U^\sigma\), and fix
		three distinct leaves \(i,j,k\). Let $v$ be the LCA of all three leaves. Since $v$ is the LCA, the three leaves cannot lie all in the same child subtree of $v$. Hence, either every leaf lies in a different subtree, and thus the three pairwise lowest common
		ancestors coincide, or two leaves lie in a subtree and one in another, in which case two pairwise common ancestors coincide and are ancestors of
		the third. Since the internal heights strictly decrease away from the root,
		either \(H_{ij}=H_{ik}=H_{jk}\), or exactly two values are equal and
		strictly larger than the third. Thus \(H\in\mathfrak U_d\). This proves
		\eqref{eq:ultrametric-stratification}. 
		
		We divide the remainder of the proof by items.
		
		(ii). The set \(D_\sigma(B):=\{(r,\eta)\in \ovU{\sigma} : \eta_v=0 \text{ for } v \in B\}\) is a face of \(\ovU{\sigma}\). Faces are preserved under \(L_\sigma:\R^{K(\sigma)}\to  V_\sigma\) as it is a linear isomorphism.
		If $(r,\eta)\in D_\sigma(B)$, then $\eta_v=0$ for every $v\in B$, and hence
		\(
		L_\sigma(r,\eta)
		=
		r\mathbf 1
		-
		\sum_{v\in I_\sigma\setminus B}
		\eta_v b_{A_v}.
		\)
		By definition, the non-root internal clusters of $\sigma/B$ are precisely
		$\{A_v:v\in I_\sigma\setminus B\}$. Therefore, by Proposition~\ref{prop:linear-hierarchy-cell},
		\(
		L_\sigma(D_\sigma(B))
		=
		\mathbb R\mathbf 1
		+
		\operatorname{cone}
		\{-b_{A_v}:v\in I_\sigma\setminus B\}
		=
		\overline{\mathfrak U}_{\sigma/B}.
		\)
		Finally, every face of \(\ovU\sigma\) is of the form \(D_\sigma(B)\) for some \(B \subset I_\sigma\) and \(L_\sigma\) gives a one-to-one correspondence between faces of \(\ovU\sigma\) and faces of \(\MovU\sigma\).
		
		(i). Take $H=L_\sigma(r,\eta) \in \MovU{\sigma}$. Set \(B(H):=\{v \in I_\sigma:\eta_v=0\}.\) By (ii), \(H \in \MovU{\gamma}\), with \(\gamma=\sigma/B(H)\). Moreover, \(H \in \MU{\gamma}\), because \(\eta_w>0\) for all \(w \in I_\gamma\).		
		Conversely, if \(\gamma \preceq \sigma\), then there is a set \(B \subset I_\sigma\) such that \(\gamma=\sigma /B\). By (ii), \(\MU{\gamma}\subset \MovU{\gamma} \subset \MovU{\sigma}\). 
		
		(iii). If
		\(H\in\MovU{\gamma}\cap
		\overline{\mathfrak U}^{\,\tau}\), by (i) we get \(\sigma(H) \preceq \gamma\) and \(\sigma(H) \preceq \tau\). Therefore \(\mathscr A(\sigma(H))\subset \mathscr A(\gamma)\cap \mathscr A(\tau)=\mathscr A(\gamma \wedge \tau)\). This gives one inclusion. Conversely, \(\gamma\wedge\tau \preceq \gamma\) and \(\gamma\wedge\tau \preceq \tau\). Thus, by (i), \(\MovU{\gamma\wedge\tau} \subset \MovU\gamma \cap \MovU\tau\).
		
		(iv).
		Parts (ii) and (iii) show that the collection is closed under taking
		faces and that intersections are common faces. It is therefore a finite
		polyhedral fan. The dimension of the cone indexed by \(\sigma\) is
		\(K(\sigma)\), so the maximal cones are exactly the binary ones. Its
		support is \(\mathfrak U_d\) by
		\eqref{eq:ultrametric-stratification}; being a finite union of closed
		polyhedral cones, it is closed.
	\end{proof}
	
	For \(H\in\mathfrak U_d\), we write \(\sigma(H)\) for its unique reduced hierarchy and \((r(H),\eta(H))\in U^{\sigma(H)}\) for the corresponding
	unique coordinates and endow \(\mathfrak{U}_d\) with the Euclidean distance \(d_E(H,H'):=\|H-H'\|_2\) inherited from the ambient space \(\R^N\). Theorem \ref{thm:ultrametric-polyhedral-characterization} also gives the following facts,
	which will be used below.
	
	\begin{corollary}
				\label{cor:closed-ultrametric-image}
		The following statements hold.
		\begin{enumerate}[label=(\roman*),ref=(\roman*)]
			\item  \label{it:cor1-it2} Each \(\mathfrak U^\sigma\) is a
			smooth embedded manifold of dimension \(K(\sigma)\);
			\item \label{it:cor1-it3}\(\mathfrak U^\gamma\subset
			\overline{\mathfrak U^\sigma}\setminus\mathfrak U^\sigma
			\Leftrightarrow
			\gamma\prec\sigma;\)
			\item \label{it:cor1-it4}every \(H\in\mathbb R^{N}\) has at least one nearest
			point in \(\mathfrak U_d\).
		\end{enumerate}
	\end{corollary}
	
	\begin{proof}
		We prove the statements point by point.
		
		
		(i). Since \(L_\sigma:\R^{K(\sigma)}\to  V_\sigma\) is a linear isomorphism, \(L_\sigma|_{U^\sigma}:U^\sigma \to \mathfrak{U}^\sigma\) is a diffeomorphism. Thus, \(\MU\sigma\) is an open subset of the \(K(\sigma)\)-dimensional linear space \( V_\sigma\).
		
		(ii) follows from Theorem \ref{thm:ultrametric-polyhedral-characterization}, (i).
		
		(iii). Since
		\(\mathfrak U_d\) is closed in a finite-dimensional Euclidean space, the
		continuous coercive function \(G\mapsto\|H-G\|_2\) attains its minimum on
		\(\mathfrak U_d\).
	\end{proof}
	
We next record the local polyhedral geometry that will be used in the
asymptotic analysis. By Proposition~\ref{prop:linear-hierarchy-cell},
the tangent space of the stratum \(\mathfrak U^\sigma\) is
\[
V_\sigma
=
\Span\left(
\mathbf 1,
\{-b_A:A\in\mathscr A(\sigma)\setminus\{[d]\}\}
\right).
\]
	If \(\sigma\preceq\tau\), set
	\begin{equation}
		\label{eq:relative-refinement-cone}
		\mathcal T_{\tau\mid\sigma}
		:=V_\sigma+
		\operatorname{cone}\left\{
		-b_A:
		A\in\mathscr A(\tau)\setminus\mathscr A(\sigma)
		\right\},
	\end{equation}
	and observe that
	\[\Span(\cl T_{\tau\mid\sigma})=\Span(\MovU\tau)=L_\tau(\R^{K(\tau)})=V_\tau.\]
	For a closed set \(F\subset\R^N\) and \(x\in F\), we use the Bouligand
	tangent cone
	\[
	\operatorname{Tan}_x(F)
	:=\left\{
	u\in\R^N:
	\text{there exist }t_n\downarrow0\text{ and }u_n\to u
	\text{ with }x+t_nu_n\in F
	\right\}.
	\]
	Moreover, we say that \(F\) is
	\emph{Chernoff regular at \(x\)} if, for every
	\(u\in\operatorname{Tan}_x(F)\), there exist \(\varepsilon>0\) and a
	continuous path \(\gamma:[0,\varepsilon)\to F\) such that
	\[
	\gamma(0)=x,
	\qquad
	\lim_{t\downarrow0}\frac{\gamma(t)-x}{t}=u.
	\]
	
	We observe that, if \(H\in\mathfrak U^\sigma\), finiteness of the fan and
	Theorem~\ref{thm:ultrametric-polyhedral-characterization} imply that there
	exists \(\delta>0\) such that
	\begin{equation}
		\label{eq:local-refinement-union}
		\mathfrak U_d\cap B_\delta(H)
		=
		\bigcup_{\tau:\,\sigma\preceq\tau}
		\left(
		\overline{\mathfrak U}^{\,\tau}\cap B_\delta(H)
		\right).
	\end{equation}
	Thus, locally around a point of \(\mathfrak U^\sigma\), only refinements of
	\(\sigma\) occur.

	\begin{lemma}
		\label{lem:local-tangent-fan}
		Let \(H\in\mathfrak U^{\sigma}\). For every \(R<\infty\)
		there exists \(t_R>0\) such that, simultaneously for every refinement
		\(\tau\succeq\sigma\),
		\begin{equation}
			\label{eq:single-refinement-exact-local-rescaling}
		t^{-1}
		\left(
		\overline{\mathfrak U}^{\,\tau}-H
		\right)
		\cap\overline B_R(0)
		=
		\mathcal T_{\tau\mid\sigma}
		\cap\overline B_R(0),
		\qquad 0<t<t_R.
		\end{equation}
		Consequently,
		\begin{equation}
			\label{eq:exact-local-tangent-rescaling}
			t^{-1}(\mathfrak U_d-H)\cap\overline B_R(0)
			=
			\mathcal T_{\sigma}\cap\overline B_R(0),
			\qquad 0<t<t_R,
		\end{equation}
		and
		\begin{equation}
			\label{eq:tangent-fan-formula}
			\operatorname{Tan}_{H}(\mathfrak U_d)
			=
			\mathcal T_{\sigma}
			:=
			\bigcup_{\tau\in\Hd:\,\sigma\preceq\tau}
			\mathcal T_{\tau\mid\sigma}.
		\end{equation}
		The set \(\mathcal T_{\sigma}\) is a finite polyhedral fan with common
		lineality space \(V_{\sigma}\). 
	\end{lemma}
	
	\begin{proof}
		Fix \(R<\infty\) and let \(\delta>0\) be as in
		\eqref{eq:local-refinement-union}. Fix a refinement \(\tau\) of \(\sigma\). In the closed
		\(\tau\)-coordinates, write
		\(
		H=L_\tau(r,\eta^{\tau}).
		\)
		We observe that the coordinates \(\eta_v\) corresponding to clusters \(A_v \in \mathscr{A}(\sigma)\setminus\{[d]\}\)
		are strictly positive, whereas those corresponding to \(\mathscr{A}(\tau)\setminus \mathscr{A}(\sigma)\) are zero. If
		\(u\in V_\tau\), write
		\(
		L_\tau^{-1}(u)=(a,\xi).
		\)
		Then, by linearity, we have
		\(
		H+tu
		=
		L_\tau(r+ta,\eta^{\tau}+t\xi),
		\)
		and hence \(H+tu \in \MovU\tau\) if and only if
		\begin{equation}
			\label{eq:coordinate-condition}
			\eta^{\tau}_v+t\xi_v \ge 0, \quad \text{for all }v \in I_\tau.
		\end{equation}		
		Since
		\(L_\tau^{-1}\) is continuous, there is
		\(M_{\tau,R}<\infty\) such that
		\(
		|\xi_v|\le M_{\tau,R}
		\)
		for every
		\(u\in V_\tau\cap\overline B_R(0)\) and every \(v\in I_\tau\).
		Since \(\eta^\tau_v>0\) for \(A_v\in\mathscr A(\sigma)\setminus\{[d]\}\), there exists \(t_{R,\tau}>0\)
		such that, whenever \(0<t<t_{R,\tau}\),
		\(
		\eta^\tau_v+t\xi_v>0
		\).
		For \(A_v\in
		\mathscr A(\tau)\setminus\mathscr A(\sigma)\), however,
		\(\eta^\tau_v=0\), so
		\eqref{eq:coordinate-condition} reduces to
		\(
		\xi_v\ge0.
		\)
		In the \(\tau\)-coordinates, these are precisely the inequalities
		defining
		\(
		\mathcal T_{\tau\mid\sigma}
		\).
		Therefore, for every
		\(u\in V_\tau\cap\overline B_R(0)\) and
		\(0<t<t_{R,\tau}\),
		\begin{equation}
			\label{eq:iff-relation}
			H+tu\in\MovU\tau \text{ if and only if } 	u\in\mathcal T_{\tau\mid\sigma}.
		\end{equation}
		There are only finitely many refinements of \(\sigma\). Hence we may
		choose
		\(
		t_R
		<
		\min\{
		\delta/(R+1),
		\min_{\tau:\,\sigma\preceq\tau}t_{R,\tau}\}.
		\)
		\eqref{eq:iff-relation} then holds simultaneously for every refinement
		\(\tau\succeq\sigma\) and every \(0<t<t_R\), proving \eqref{eq:single-refinement-exact-local-rescaling}.
		
		Moreover, \(R\, t<\delta\). Hence,
		\eqref{eq:local-refinement-union} and \eqref{eq:single-refinement-exact-local-rescaling} give
		\[
		\begin{aligned}
			t^{-1}(\mathfrak U_d-H)\cap\overline B_R(0)
			&=
			\bigcup_{\tau:\,\sigma\preceq\tau}
			\left[
			t^{-1}
			\left(
			\overline{\mathfrak U}^{\,\tau}-H
			\right)
			\cap\overline B_R(0)
			\right]=
			\bigcup_{\tau:\,\sigma\preceq\tau}
			\left(
			\mathcal T_{\tau\mid\sigma}
			\cap\overline B_R(0)
			\right)=
			\mathcal T_\sigma\cap\overline B_R(0).
		\end{aligned}
		\]
		This proves \eqref{eq:exact-local-tangent-rescaling}, and
		\eqref{eq:tangent-fan-formula} follows from the definition of the
		Bouligand tangent cone.
		Finally, the cones in \eqref{eq:tangent-fan-formula} are indexed by the refinements
		\(\tau\) of \(\sigma\), equivalently by the
		cones of \(\mathcal F_d\) containing the face
		\(\overline{\mathfrak U}_\sigma\). Since faces and intersections among these cones
		are inherited from the finite fan \(\mathcal F_d\), they form a finite
		polyhedral fan.
	\end{proof}
	
\begin{corollary}
	\label{cor:chernoff-regularity}
	Let \(H\in\mathfrak U^\sigma\). Then \(\mathfrak U_d\) is
	Chernoff regular at \(H\).
\end{corollary}

\begin{proof}
	Let
	\(u\in\mathcal T_\sigma\), and choose \(R>\|u\|_2\).
	By \eqref{eq:exact-local-tangent-rescaling}, there exists
	\(t_R>0\) such that
	\[
	t^{-1}(\mathfrak U_d-H)\cap B_R(0)
	=
	\mathcal T_\sigma\cap B_R(0),
	\qquad 0<t<t_R.
	\]
	Since \(u\in\mathcal T_\sigma\cap B_R(0)\), it follows that
	\(
	H+tu\in\mathfrak U_d\) for \(0<t<t_R.
	\)
	Thus the path \(\gamma(t)=H+tu\) is feasible for all sufficiently
	small \(t\ge0\) and satisfies
	\((\gamma(t)-H)/t=u.
	\)
	Hence \(\mathfrak U_d\) is Chernoff regular at \(H\).
\end{proof}

	\section{Homogeneous Archimedean realization}
	\label{sec:archimedean-realization}
	
	We now introduce a homogeneous Archimedean family and realize the
	ultrametric hierarchy space \(\mathfrak U_d\) as a class of nested copulas.
	
	\begin{assumption}{K}
		\label{ass:K}
		Let
		\(
		\mathcal K=\{C_\theta^{(m)}:\theta\in\Theta,\ m=2,\ldots,d\}
		\)
		be a one-parameter homogeneous Archimedean family, where
		\(
		C_\theta^{(m)}(u_1,\ldots,u_m)
		:=
		\psi_\theta(
		\psi_\theta^{-1}(u_1)+\cdots+\psi_\theta^{-1}(u_m)
		).
		\)
		We assume:
		\begin{enumerate}[label=\emph{(K\arabic*)}, leftmargin=3em]
			\item\label{it:k1} For every \(\theta\in\Theta\) and every \(m=2,\ldots,d\),
			\(C_\theta^{(m)}\) is an \(m\)-variate copula.
			\item \label{ass:bivariate-identifiability} The bivariate family is
			topologically identifiable: 
			\(\|C_{\theta_n}^{(2)}-C_\theta^{(2)}\|_\infty\to0\) implies
			\(\theta_n\to\theta\).
			\item \label{it:continuous-family} For every \(m=2,\ldots,d\),
			\(\theta_n\to\theta\) implies
			\(\|C_{\theta_n}^{(m)}-C_\theta^{(m)}\|_\infty\to0\).
		\end{enumerate}
	\end{assumption}

	Given a hierarchy \(\sigma\) and a parameter vector
	\((\theta_v)_{v\in\Int(\sigma)}\), define the copula function
	\(C_{\sigma,\theta}\) recursively. If \(i\) is a leaf, set
	\(C_i(u):=u_i\). If \(v\) has children \(w_1,\ldots,w_m\), define
	\begin{equation}
		\label{eq:recursive-construction}
		C_v(u)
		:=C_{\theta_v}^{(m)}
		\bigl(C_{w_1}(u),\ldots,C_{w_m}(u)\bigr).
	\end{equation}
	The function represented by the decorated hierarchy is the value \(C_\rho\)
	at the root.
	
	The recursive construction does not automatically define a copula for an arbitrary assignment of parameters to the internal vertices. Although every node function \(C_{\theta_v}^{(m)}\) is itself a copula, inserting a child copula into its parent need not preserve \(d\)-increasingness. The resulting function may therefore fail to be a multivariate distribution function. A nesting condition is imposed precisely to prevent this failure: it ensures that every parent--child composition occurring in the recursion is admissible and hence that the value \(C_\rho\) at the root is a \(d\)-dimensional copula. In the homogeneous one-parameter setting considered here, we encode this condition as follows.

	\begin{assumption}{H}
		\label{ass:kendall-height}
		Let
		\(
		\tilde\tau_{\mathcal K}(\theta)
		:=
		\tau\!\left(C_\theta^{(2)}\right),
		\)
		where \(\tau\) denotes Kendall's tau. We assume that there exists
		\(\tau_{\mathcal K}^\star\in(0,1]\) such that
		\(
		\tilde\tau_{\mathcal K}:
		\Theta\to(0,\tau_{\mathcal K}^\star)
		\)
		is a strictly increasing homeomorphism.
	\end{assumption}
	\begin{assumption}{N}
		\label{ass:nesting}
		The homogeneous parameter order
		\(
		\theta_{\parv(v)}\le \theta_v\) for \(v\in\Iint,
		\)
		is a sufficient nesting condition for the recursion
		\eqref{eq:recursive-construction}.
	\end{assumption}

		Define the normalized Kendall coordinate and its logit height by
		\begin{equation}
			\label{eq:kendall-height}
			\kappa_{\mathcal K}(\theta)
			:=
			\frac{\tilde\tau_{\mathcal K}(\theta)}
			{\tau_{\mathcal K}^\star},
			\qquad
			h_{\mathcal K}(\theta)
			:=
			\log
			\frac{1-\kappa_{\mathcal K}(\theta)}
			{\kappa_{\mathcal K}(\theta)}
			=
			\log
			\frac{
				\tau_{\mathcal K}^\star-\tilde\tau_{\mathcal K}(\theta)
			}{
				\tilde\tau_{\mathcal K}(\theta)
			}.
		\end{equation}
		Then \(h_{\mathcal K}:\Theta\to\mathbb R\) is a strictly decreasing
		homeomorphism. We write
		\begin{equation}
			\label{eq:theta-from-kendall-height}
			\theta_{\mathcal K}(h)
			:=
			h_{\mathcal K}^{-1}(h)
			=
			\tilde\tau_{\mathcal K}^{-1}
			\left(
			\frac{\tau_{\mathcal K}^\star}{1+e^h}
			\right),
			\qquad h\in\mathbb R.
		\end{equation}
		In particular, the nesting condition is equivalently
		\(
			h_{\mathcal K}(\theta_{\parv(v)})
			\ge
			h_{\mathcal K}(\theta_v),\) for all \( v\in\Iint.
		\)
		Let
		\(
		D_1(\theta)
		:=
		1/\theta
		\int_0^\theta t/(e^t-1)\,dt
		\)
		denote the Debye function of order one.

	\begin{proposition}
		\label{lem:admissible-families}
		The families listed in the Table \ref{tab:kendall-profiles} below satisfy
		Assumptions~\assref{ass:K},~\assref{ass:kendall-height} and \assref{ass:nesting}.
		The last column gives the upper endpoint
		\(\tau_{\mathcal K}^\star\) of the positive-dependence Kendall range.
		\begin{table}[pos=htbp]
			\centering
			\small
			\setlength{\tabcolsep}{4pt}
			\renewcommand{\arraystretch}{1.15}
			\begin{tabular}{@{}lcccc@{}}
				\toprule
				Family \(\mathcal K\)
				&
				\(\Theta\)
				&
				\(\psi_\theta(u)\)
				&
				\(\tilde\tau_{\mathcal K}(\theta)
				=\tau(C^{(2)}_\theta)\)
				&
				\(\tau^\star_{\mathcal K}\)
				\\
				\midrule
				Clayton
				&
				\((0,\infty)\)
				&
				\((1+u)^{-1/\theta}\)
				&
				\(\displaystyle \frac{\theta}{\theta+2}\)
				&
				\(1\)
				\\[1.2ex]
				
				Frank
				&
				\((0,\infty)\)
				&
				\(\displaystyle
				-\frac{1}{\theta}
				\log\!\left(
				1-(1-e^{-\theta})e^{-u}
				\right)\)
				&
				\(\displaystyle
				1+\frac{4}{\theta}
				\bigl(D_1(\theta)-1\bigr)\)
				&
				\(1\)
				\\[1.2ex]
				
				AMH
				&
				\((0,1)\)
				&
				\(\displaystyle
				\frac{1-\theta}{e^u-\theta}\)
				&
				\(\displaystyle
				1-
				\frac{
					2\{\theta+(1-\theta)^2\log(1-\theta)\}
				}{
					3\theta^2
				}\)
				&
				\(\displaystyle \frac13\)
				\\[1.2ex]
				
				Gumbel
				&
				\((1,\infty)\)
				&
				\(\exp(-u^{1/\theta})\)
				&
				\(\displaystyle 1-\frac1\theta\)
				&
				\(1\)
				\\[1.2ex]
				
				Joe
				&
				\((1,\infty)\)
				&
				\(1-(1-e^{-u})^{1/\theta}\)
				&
				\(\displaystyle
				1-
				4\sum_{k=1}^{\infty}
				\frac{1}{
					k(\theta k+2)(\theta(k-1)+2)
				}\)
				&
				\(1\)
				\\
				\bottomrule
			\end{tabular}
			\caption{Admissible homogeneous Archimedean families, their generators,
				Kendall coefficients, and positive-dependence ranges.}
			\label{tab:kendall-profiles}
		\end{table}
	\end{proposition}

	\begin{proof}
		For each family in the table, the standard Archimedean generator is
		completely monotone on the indicated parameter space and therefore
		generates an \(m\)-variate copula for every \(m\ge2\)
		\cite{Joe1997,McNeilNeslehova2009}. Hence
		Assumption~\ref{it:k1} holds.

		We next verify Assumption~\ref{it:continuous-family}. Fix
		\(m\in\{2,\ldots,d\}\) and let \(\theta_n\to\theta\) in the
		relevant parameter space. The generators and their inverses depend
		continuously on the parameter, and therefore
		\(
		C_{\theta_n}^{(m)}(u)\to C_\theta^{(m)}(u)
		\)
		pointwise on \((0,1]^m\). At points with a zero coordinate both sides
		vanish. Every \(m\)-variate copula satisfies
		\(
		|C(u)-C(v)|
		\le
		\sum_{j=1}^m|u_j-v_j|,
		\)
		so the sequence is uniformly equicontinuous. Since \([0,1]^m\) is
		compact and \(C_\theta^{(m)}\) is continuous, pointwise convergence
		upgrades to uniform convergence.

		For the five families, the formulas for
		\(\tilde\tau_{\mathcal K}(\theta)=\tau(C^{(2)}_\theta)\)
		are given in Table~\ref{tab:kendall-profiles}. Derivative computations show that
		\(\tilde\tau_{\mathcal K}'(\theta)>0\) throughout the corresponding
		parameter interval. Together with the endpoint limits, this gives
		\(
		\tilde\tau_{\mathcal K}(\Theta)
		=
		(0,\tau_{\mathcal K}^\star),
		\)
		with the values of \(\tau_{\mathcal K}^\star\) stated in the table.
		Thus
		Assumption~\assref{ass:kendall-height} holds, and
		\(h_{\mathcal K}\) in \eqref{eq:kendall-height} is a strictly
		decreasing homeomorphism onto \(\mathbb R\).

		To verify Assumption~\ref{ass:bivariate-identifiability}, suppose
		\(
		\|C_{\theta_n}^{(2)}-C_\theta^{(2)}\|_\infty
		\to0.
		\)
		Let \(\mu_n\) and \(\mu\) be the probability measures induced by
		\(C_{\theta_n}^{(2)}\) and \(C_\theta^{(2)}\), respectively.
		Uniform convergence implies \(\mu_n\Rightarrow\mu\). Using
		\(
		\tau(C)=4\int_{[0,1]^2}C\,d\mu_C-1,
		\)
		we obtain
		\[
		\begin{aligned}
			\left|
			\int C_{\theta_n}^{(2)}\,d\mu_n
			-
			\int C_\theta^{(2)}\,d\mu
			\right|
			&\le
			\left\|C_{\theta_n}^{(2)}-C_\theta^{(2)}\right\|_\infty
			+
			\left|
			\int C_\theta^{(2)}\,d\mu_n
			-
			\int C_\theta^{(2)}\,d\mu
			\right|
			\to0.
		\end{aligned}
		\]
		Hence
		\(
		\tilde\tau_{\mathcal K}(\theta_n)
		\to
		\tilde\tau_{\mathcal K}(\theta).
		\)
		Since \(\tilde\tau_{\mathcal K}\) is a homeomorphism onto its image,
		\(\theta_n\to\theta\), proving
		Assumption~\ref{ass:bivariate-identifiability}.

		Finally, the standard sufficient nesting criterion for two completely
		monotone Archimedean generators
		\(\psi_{\theta_0}\) and \(\psi_{\theta_1}\) requires
		\(
		\psi_{\theta_0}^{-1}\circ\psi_{\theta_1}
		\)
		to have a completely monotone derivative. For homogeneous parent--child
		pairs from any of the five families, this criterion holds whenever
		\(
		\theta_0\le\theta_1
		\), \cite{HofertMaechler2011,McNeil2008}, giving Assumption \assref{ass:nesting}.
	\end{proof}

	Let \(\mathcal C_d^{\mathcal K}\) be the set of all \(d\)-dimensional
	copulas produced by the recursion~\eqref{eq:recursive-construction} from a
	hierarchy in \(\Hd\) and a parameter vector satisfying the nesting condition in
	\assref{ass:nesting}.	
	For \(H\in\mathfrak U_d\), define
	\[
	\theta_v(H):=
	\theta_{\mathcal K}\!\left(h_v(r(H),\eta(H))\right),
	\qquad v\in\Int(\sigma(H)).
	\]
	Since \(\eta_v>0\), the nesting condition is satisfied strictly. We define
	\[
	R_{\mathcal K}:\mathfrak U_d\to\mathcal C_d^{\mathcal K},
	\qquad
	R_{\mathcal K}(H):=C_{\sigma(H),\theta(H)}.
	\]
	Symmetry of Archimedean copulas makes this construction independent of any
	ordering of the children.
	
	\begin{proposition}
		\label{lem:equal-parameter-contraction}
		Assume~\assref{ass:K}. If a parent internal vertex and one of its
		internal children carry the same parameter \(\theta\), then contracting
		the edge between them does not change the represented copula.
	\end{proposition}
	
	\begin{proof}
		Let \(v\) be an internal vertex and \(w\) one of its internal children,
		and assume that both carry \(\theta\). Let the other children of \(v\)
		contribute \(z_1,\ldots,z_m\), and let the children of \(w\) contribute
		\(y_1,\ldots,y_n\). Then
		\[
		\begin{aligned}
			C_v
			&=
			\psi_\theta\left(
			\sum_{i=1}^m\psi_\theta^{-1}(z_i)
			+
			\psi_\theta^{-1}\left[
			\psi_\theta\left(
			\sum_{j=1}^n\psi_\theta^{-1}(y_j)
			\right)
			\right]
			\right)\\
			&=
			\psi_\theta\left(
			\sum_{i=1}^m\psi_\theta^{-1}(z_i)
			+
			\sum_{j=1}^n\psi_\theta^{-1}(y_j)
			\right)\\
			&=
			C_\theta^{(m+n)}(z_1,\ldots,z_m,y_1,\ldots,y_n).
		\end{aligned}
		\]
		All other subtree values are unchanged, so the value at the root is
		unchanged.
	\end{proof}

	The next theorem shows that the ultrametric vector is the unique reduced
	parameter of a homogeneous nested Archimedean copula. Assumptions
	\assref{ass:K}, \assref{ass:kendall-height} and \assref{ass:nesting} will be in force throughout the section.
	
	\begin{theorem}
		\label{thm:identifiability}
		\(R_{\mathcal K}:\mathfrak U_d\to\mathcal C_d^{\mathcal K}\) is a
		bijection.
	\end{theorem}
	
	\begin{proof}
		We first prove surjectivity. Let \(C\in\mathcal C_d^{\mathcal K}\). By
		definition, \(C\) is represented by a hierarchy \(\sigma\in\Hd\) and
		parameters \((\theta_v)_{v\in\Int(\sigma)}\) satisfying the nesting
		condition. Set \(h_v:=h_{\mathcal K}(\theta_v)\) and
		\(\eta_v:=h_{\parv(v)}-h_v\ge0\).
		
		Contract every internal edge with \(\eta_v=0\). Along such an edge the
		endpoint parameters coincide, so Proposition~\ref{lem:equal-parameter-contraction}
		shows that the represented copula is unchanged. After finitely many
		contractions, we obtain a reduced hierarchy \(\sigma'\) with strictly
		positive increments \(\eta'\). Let
		\(H:=L_{\sigma'}(h_\rho,\eta')\in\mathfrak U^{\sigma'}\). By construction,
		\(C=R_{\mathcal K}(H)\), proving surjectivity.
		
		For injectivity, let \(H,H'\in\mathfrak U_d\), with reduced hierarchies
		\(\sigma,\sigma'\), and assume
		\(R_{\mathcal K}(H)=R_{\mathcal K}(H')\). Fix \(i<j\). Setting all
		variables except \(u_i,u_j\) equal to one in the recursive construction
		shows that the corresponding bivariate margins are
		\(
		C_{\theta_{\mathcal K}(H_{ij})}^{(2)}\) and \(
		C_{\theta_{\mathcal K}(H'_{ij})}^{(2)}
		\) respectively.
		Their Kendall coefficients are therefore
		\(
		\tau_{\mathcal K}^\star/(1+\exp(H_{ij}))\) and \(
		\tau_{\mathcal K}^\star/(1+\exp(H'_{ij})).
		\) Equality of the copulas implies equality of these
		coefficients, and strict monotonicity gives \(H_{ij}=H'_{ij}\).
		This holds for every pair, so \(H=H'\).
	\end{proof}
	
	\subsection{Geometry of the realized copula class}
	\label{sec:copula-geometry}
	
	For \(\sigma\in\Hd\), define
	\(
	\mathcal C_\sigma^{\mathcal K}
	:=R_{\mathcal K}(\mathfrak U^\sigma).
	\)
	Theorem~\ref{thm:identifiability} gives the disjoint decomposition
	\(\mathcal C_d^{\mathcal K}=\bigsqcup_{\sigma\in\Hd}
	\mathcal C_\sigma^{\mathcal K}\). The stratum
	\(\mathcal C_\sigma^{\mathcal K}\) consists precisely of the copulas whose
	reduced hierarchy is \(\sigma\).
	
	Let \(C\in\mathcal C_d^{\mathcal K}\), and let \(C_{ij}\) denote its
	bivariate margin on coordinates \((i,j)\). By
	Assumption~\ref{ass:bivariate-identifiability}, there is a unique
	\(\theta_{ij}(C)\in\Theta\) such that
	\begin{equation}
		\label{eq:bivariate-margins}
		C_{ij}=C_{\theta_{ij}(C)}^{(2)}.
	\end{equation}
	The height of this margin is
	\(
	h_{\mathcal K}(\theta_{ij}(C))
	=
	\log(
	(\tau_{\mathcal K}^\star-\tau(C_{ij})
	)/\tau(C_{ij})).
	\)
	We therefore define the cophenetic map by
	\begin{equation}
		\label{eq:cophenetic-map-kendall}
		\Psi_{\mathcal K}:\mathcal C_d^{\mathcal K}\to\mathbb R^{N},
		\qquad
		\Psi_{\mathcal K}(C)_{ij}
		:=
		\log
		\frac{
			\tau_{\mathcal K}^\star-\tau(C_{ij})
		}{
			\tau(C_{ij})
		}.
	\end{equation}
	For \(C,C'\in\mathcal C_d^{\mathcal K}\), set
	\(
	d_{\mathcal C}(C,C')
	:=\|\Psi_{\mathcal K}(C)-\Psi_{\mathcal K}(C')\|_2.
	\)
	
	\begin{proposition}
		\label{prop:dC-metric} 
		The map \(\Psi_{\mathcal K}\) is a bijection from
		\(\mathcal C_d^{\mathcal K}\) onto \(\mathfrak U_d\), with inverse
		\(R_{\mathcal K}\). In particular, \(d_{\mathcal C}\) is a metric and
		\(
		R_{\mathcal K}:(\mathfrak U_d,d_E)\to
		(\mathcal C_d^{\mathcal K},d_{\mathcal C})
		\)
		is an isometry and a homeomorphism.
	\end{proposition}
	
	\begin{proof}
		Let \(H\in\mathfrak U_d\). The \((i,j)\)-margin of
		\(R_{\mathcal K}(H)\) is
		\(C_{\theta_{\mathcal K}(H_{ij})}^{(2)}\). Its Kendall coefficient is
		\(\tau_{\mathcal K}^\star/(1+\exp(H_{ij}))\). Therefore
		\begin{equation}
			\label{eq:Psi-composition}
			\Psi_{\mathcal K}(R_{\mathcal K}(H))=H.
		\end{equation}
		Since \(R_{\mathcal K}\) is bijective, this identity proves that
		\(\Psi_{\mathcal K}=R_{\mathcal K}^{-1}\) and has image
		\(\mathfrak U_d\). The metric and isometry statements follow directly. A bijective isometry is a homeomorphism.
	\end{proof}
	
	We call \(d_{\mathcal C}\) the cophenetic metric.
%
	 The restriction
	\(L_\sigma:U^\sigma\to\MU\sigma\) is linear and bijective with linear inverse and, thus, is a diffeomorphism. Therefore \(\MU\sigma\) is a \(K(\sigma)\)-dimensional smooth manifold embedded in \(\R^N\). We give
	\(\mathcal C_\sigma^{\mathcal K}\) the unique smooth structure for which
	\[
	\Psi_{\mathcal K}^\sigma
	:=\Psi_{\mathcal K}|_{\mathcal C_\sigma^{\mathcal K}}:
	\mathcal C_\sigma^{\mathcal K}\to\mathfrak U^\sigma
	\]
	is a diffeomorphism.
	
	\begin{theorem}
		\label{thm:stratified-isomorphism}
		The copula class \(\mathcal C_d^{\mathcal K}\) is stratified by the smooth
		manifolds \(\mathcal C_\sigma^{\mathcal K}\), \(\sigma\in\Hd\), with
		\(\dim\mathcal C_\sigma^{\mathcal K}=K(\sigma)\). The realization map \(R_{\cl K}\)
		maps \(\mathfrak U^\sigma\) diffeomorphically onto
		\(\mathcal C_\sigma^{\mathcal K}\), and
		\[
		\mathcal C_\gamma^{\mathcal K}
		\subset
		\overline{\mathcal C_\sigma^{\mathcal K}}
		\setminus\mathcal C_\sigma^{\mathcal K}
		\quad\Longleftrightarrow\quad
		\gamma\prec\sigma.
		\]
	\end{theorem}
	
	\begin{proof}
		The diffeomorphism statement follows from the definition of the smooth
		structure and Proposition~\ref{prop:dC-metric}. The dimension and frontier
		relation are transported through the homeomorphism \(R_{\mathcal K}\) from
		Corollary~\ref{cor:closed-ultrametric-image}.
	\end{proof}
	
	\begin{proposition}
		\label{prop:uniform-topology-equivalence}
		For every sequence
		\((C_n)_{n\ge1}\subset\mathcal C_d^{\mathcal K}\) and every
		\(C\in\mathcal C_d^{\mathcal K}\),
		\begin{equation}
			\label{eq:metric-equivalence}
			d_{\mathcal C}(C_n,C)\to0
			\quad\Longleftrightarrow\quad
			\|C_n-C\|_\infty\to0.
		\end{equation}
		Consequently, the cophenetic metric induces exactly the topology inherited
		from \(C([0,1]^d)\) through the uniform norm.
	\end{proposition}
	
	\begin{proof}
		Suppose first that \(d_{\mathcal C}(C_n,C)\to0\), and set
		\(H_n:=\Psi_{\mathcal K}(C_n)\) and
		\(H:=\Psi_{\mathcal K}(C)\). Then \(H_n\to H\) in
		\(\mathbb R^{N}\). We show that
		\(R_{\mathcal K}(H_n)\to R_{\mathcal K}(H)\) uniformly.
		
		It is enough to prove this along an arbitrary subsequence. Since \(\Hd\)
		is finite, after passing to a further subsequence we may assume that
		\(H_n\in\mathfrak U^\sigma\) for one fixed hierarchy \(\sigma\). Write
		\(H_n=L_\sigma(x_n)\), with \(x_n\in U^\sigma\). Since
		\(H\in\MovU{\sigma}\), there is a unique
		\(x\in\ovU{\sigma}\) such that \(H=L_\sigma(x)\), and
		continuity of the linear inverse gives \(x_n\to x\).
		For every internal vertex \(v\), the corresponding parameters satisfy
		\(
		\theta_v(x_n)=\theta_{\mathcal K}(h_v(x_n))
		\to
		\theta_{\mathcal K}(h_v(x))=\theta_v(x).
		\)
		
		For a vertex \(v\) of \(\sigma\), let \(C_{v,n}\) and \(C_v\) denote the
		subtree copulas represented by \((\sigma,x_n)\) and \((\sigma,x)\),
		respectively. We prove by induction from the leaves to the root that
		\(
		\|C_{v,n}-C_v\|_\infty \to 0 .
		\)
		
		If \(v=i\) is a leaf, then by definition
		\(
		C_{i,n}(u)=u_i=C_i(u),
		\)
		and therefore \(\|C_{i,n}-C_i\|_\infty=0\).
		Now let \(v\) be an internal vertex with children
		\(w_1,\ldots,w_m\), and suppose that
		\(
		\|C_{w_j,n}-C_{w_j}\|_\infty\to 0,
		\) for \(j=1,\ldots,m.\)
		Denote \(\theta_v:=\theta_v(x)\) and \(\theta_{v,n}:=\theta_v(x_n)\). Then
		\[
		\begin{aligned}
			|C_{v,n}(u)-C_v(u)|
			&\le
			\left|
			C^{(m)}_{\theta_{v,n}}\left(C_{w_1,n}(u),\ldots,C_{w_m,n}(u)\right)
			-
			C^{(m)}_{\theta_v}(C_{w_1,n}(u),\ldots,C_{w_m,n}(u))
			\right|
			\\
			&\quad+
			\left|
			C^{(m)}_{\theta_v}(C_{w_1,n}(u),\ldots,C_{w_m,n}(u))
			-
			C^{(m)}_{\theta_v}(C_{w_1}(u),\ldots,C_{w_m}(u))
			\right|.
		\end{aligned}
		\]
		The Lipschitz estimate
		\(
		\left|C(u)-C(v)\right|
		\leq
		\sum_{j=1}^m |u_j-v_j|
		\) gives
		\(
		\|C_{v,n}-C_v\|_\infty
		\le\|
		C^{(m)}_{\theta_{v,n}}
		-
		C^{(m)}_{\theta_v}\|_\infty
		+
		\sum_{j=1}^m
		\|C_{w_j,n}-C_{w_j}\|_\infty .
		\)
		The first term tends to zero by
		Assumption~\ref{it:continuous-family}, while the remaining terms tend
		to zero by the induction hypothesis. Hence
		\(
		\|C_{v,n}-C_v\|_\infty\to 0.
		\)
		Since \(\sigma\) has finitely many vertices, iterating this argument
		from the leaves to the root \(\rho\) yields
		\(
		\|C_{\rho,n}-C_\rho\|_\infty\to 0.
		\)
		Thus the copulas represented by \((\sigma,x_n)\) converge uniformly to
		the copula represented by \((\sigma,x)\). If an increment of \(x\)
		vanishes, its endpoint parameters coincide; repeated use of
		Proposition~\ref{lem:equal-parameter-contraction} shows that the limiting
		copula corresponds to \(R_{\mathcal K}(H)\). Thus every subsequence has a
		further subsequence converging uniformly to \(C\), and hence
		\(\|C_n-C\|_\infty\to0\).
		
		Conversely, suppose that \(\|C_n-C\|_\infty\to0\). Let \(C_{n,ij}\) and
		\(C_{ij}\) be the corresponding bivariate margins. Then
		\(
		\|C_{n,ij}-C_{ij}\|_\infty
		\le\|C_n-C\|_\infty\to0.
		\)
		By~\eqref{eq:bivariate-margins} and
		Assumption~\ref{ass:bivariate-identifiability},
		\(\theta_{n,ij}\to\theta_{ij}\). Continuity of
		\(h_{\mathcal K}\) gives
		\(\Psi_{\mathcal K}(C_n)_{ij}\to\Psi_{\mathcal K}(C)_{ij}\) for every
		pair. Since there are finitely many pairs,
		\(d_{\mathcal C}(C_n,C)\to0\).
	\end{proof}
	
\subsection{Pairwise copula functionals as coordinates}
\label{subsec:pairwise-functionals}

We now relate the cophenetic geometry of
\(\cl C_d^{\cl K}\) to bivariate copula functionals. The key
observation is that the collection of pairwise functional values is
obtained by applying a single scalar profile componentwise to the
cophenetic vector.

Let \(\mathfrak C_2\) denote the set of bivariate copulas, and let
\(
T:\mathfrak C_2\to\mathbb R
\)
be a bivariate copula functional. For \(C\in\cl C_d^{\cl K}\), denote
by \(C_{ij}\) its bivariate margin on the coordinates \((i,j)\), and
define
\[
T_{ij}(C):=T(C_{ij}),
\qquad
\mathsf T(C)
:=
\bigl(T_{ij}(C)\bigr)_{1\le i<j\le d}
\in\mathbb R^N.
\]
The profile of \(T\) along the bivariate family \(\cl K\) is the map
\[
q_T:\mathbb R\longrightarrow\mathbb R,
\qquad
h \mapsto
T\left(C^{(2)}_{\theta_{\mathcal K}(h)}\right).
\]
We define the map
\(
Q_T:\mathbb R^N\to\mathbb R^N,
\ 
Q_T(H)_{ij}:=q_T(H_{ij}).
\)
If \(q_T\) is injective, we denote its coordinatewise inverse by
\(
Q_T^{-1}:q_T(\mathbb R)^N\to\mathbb R^N,
\ 
Q_T^{-1}(z)_{ij}:=q_T^{-1}(z_{ij}).
\)

\begin{theorem}
	\label{thm:informative-functionals}
	For every bivariate copula functional \(T\),
	\begin{equation}
		\label{eq:T-comp}
		\mathsf T
		=
		Q_T\circ\Psi_{\cl K}.
	\end{equation}
	The following consequences hold.
	\begin{enumerate}[label=\textup{(\roman*)}]
		\item If \(q_T\) is continuous, then
		\(
		\mathsf T:
		(\cl C_d^{\cl K},d_{\cl C})
		\to
		\R^N
		\)
		is continuous.
		
		\item If \(q_T\) is injective, then \(\mathsf T\) is injective.
		
		\item If \(q_T\) is continuous and injective, then
		\(
		\mathsf T:
		\cl C_d^{\cl K}
		\to
		\mathsf T(\cl C_d^{\cl K})
		\)
		is a topological embedding. Equivalently, for every
		\((C_n)_{n\ge1}\subset\cl C_d^{\cl K}\) and
		\(C\in\cl C_d^{\cl K}\),
		\[
		d_{\cl C}(C_n,C)\longrightarrow0
		\quad\Longleftrightarrow\quad
		\|\mathsf T(C_n)-\mathsf T(C)\|_2\longrightarrow0.
		\]
		
		\item If \(q_T:\mathbb R\to q_T(\mathbb R)\) is locally
		bi-Lipschitz, then for every compact set
		\(K\subset\mathfrak U_d\) there exist constants
		\(0<c_K\le C_K<\infty\) such that
		\[
		c_Kd_{\cl C}(C,C')
		\le
		\|\mathsf T(C)-\mathsf T(C')\|_2
		\le
		C_Kd_{\cl C}(C,C')
		\]
		whenever
		\(\Psi_{\cl K}(C),\Psi_{\cl K}(C')\in K\).
	\end{enumerate}
\end{theorem}

\begin{proof}
	For \(C\in\cl C_d^{\cl K}\), its \((i,j)\)-margin is
	\(
	C_{ij}
	=
	C^{(2)}_{\theta_{\mathcal K}(\Psi_{\cl K}(C)_{ij})}.
	\)
	Therefore,
	\(
		T_{ij}(C)
		=
		q_T(\Psi_{\cl K}(C)_{ij}),
	\)
	which proves \eqref{eq:T-comp}.
	
%
	
	If \(q_T:\R \to q_T(\R)\) is both continuous and injective, it is strictly monotone and hence a
	homeomorphism. Now, (i), (ii), and (iii) follow from \eqref{eq:T-comp} and the fact that \(\Psi_{\cl K}:
	(\cl C_d^{\cl K},d_{\cl C})
	\to
	\mathfrak U_d\) is a bijective isometry and a homeomorphism.
	
	Finally, assume that \(q_T\) is locally bi-Lipschitz and let
	\(K\subset\mathfrak U_d\) be compact. The interval
	\[
	J_K
	:= \bigl[\min_{H \in K, 1\le i < j \le d} H_{ij},\max_{H \in K, 1\le i < j \le d} H_{ij}\bigr]
	\]
	is compact. By local bi-Lipschitz continuity there
	exist constants \(0<c_K\le C_K<\infty\) such that
	\(
	c_K|h-h'|
	\le
	|q_T(h)-q_T(h')|
	\le
	C_K|h-h'|,\) for \(
	h,h'\in J_K.
	\)
	Thus, whenever
	\(\Psi_{\cl K}(C),\Psi_{\cl K}(C')\in K\), one has
	\(
		c_K^2d_{\cl C}(C,C')^2
		\le
		\|\mathsf T(C)-\mathsf T(C')\|_2^2\le
		C_K^2d_{\cl C}(C,C')^2.
	\)
	Taking square roots proves \textup{(iv)}.
\end{proof}

Theorem~\ref{thm:informative-functionals} separates the geometric and statistical ingredients
of the construction.  The hierarchy geometry is carried entirely by the
cophenetic vector \(H\), whereas the choice of pairwise functional enters
only through the scalar transformation \(q_T\).  In particular, any
informative bivariate functional yields the same reduced hierarchy
coordinates after componentwise inversion.  Thus the reconstruction
problem is not intrinsically tied to Kendall's tau; Kendall's tau is one
particularly convenient realization of a more general coordinate
principle.

We observe that, if $q_T$ is injective, the cophenetic vector is recovered componentwise through
\(
	\Psi_{\cl K}(C)
	=
	Q_T^{-1}(\mathsf T(C))
	=
	\bigl(q_T^{-1}(T_{ij}(C))\bigr)_{i<j}.
\)
Consequently,
\(
C
=
R_{\cl K}\!\left(Q_T^{-1}(\mathsf T(C))\right),
\)
so the pairwise \(T\)-values determine the copula, its reduced
hierarchy, and all its node parameters.

\FloatBarrier


We now verify the hypotheses of Theorem~\ref{thm:informative-functionals}
for Spearman's rho, and the lower- and upper-tail
dependence coefficients in the admissible families. 
By construction, Kendall's tau is the reference coordinate functional. In particular, by \eqref{eq:theta-from-kendall-height},
\begin{equation}
	\label{eq:universal-kendall-profile}
	q_\tau(h)
	=
	\tau\!\left(C^{(2)}_{\theta_{\mathcal K}(h)}\right)
	=
	\frac{\tau_{\mathcal K}^\star}{1+e^h},
	\qquad h\in\mathbb R,
\end{equation}
which readily satisfies the assumptions of Theorem~\ref{thm:informative-functionals}(iii)-(iv).



The tail-dependence profiles in the new coordinates are summarized in
Table~\ref{tab:tail-profiles}.
\begin{table}[pos=htbp]
	\centering
	\small
	\setlength{\tabcolsep}{4pt}
	\renewcommand{\arraystretch}{1.15}
	\begin{tabular}{@{}lccc@{}}
		\toprule
		Family \(\cl K\)
		&
		\(q_L(h)\)
		&
		\(q_U(h)\)
		&
		Informative tail profile
		\\
		\midrule
		Clayton
		&
		\(\displaystyle 2^{-e^h/2}\)
		&
		\(0\)
		&
		lower
		\\
		
		Frank
		&
		\(0\)
		&
		\(0\)
		&
		none
		\\
		
		AMH
		&
		\(0\)
		&
		\(0\)
		&
		none
		\\
		
		Gumbel
		&
		\(0\)
		&
		\(\displaystyle 2-2^{1/(1+e^{-h})}\)
		&
		upper
		\\
		
		Joe
		&
		\(0\)
		&
		\(\displaystyle 2-2^{1/\theta_{\mathrm{Jo}}(h)}\)
		&
		upper
		\\
		\bottomrule
	\end{tabular}
	\caption{Lower- and upper-tail dependence profiles in the
		Kendall-logit cophenetic coordinate.}
	\label{tab:tail-profiles}
\end{table}
The nonconstant tail profiles in Table~\ref{tab:tail-profiles} are
smooth, strictly decreasing, and locally bi-Lipschitz. The remaining
tail profiles are identically zero and therefore non-informative.
Combining these calculations with
Theorem~\ref{thm:informative-functionals} gives the following corollary.

\begin{corollary}
	\label{cor:classical-functionals}
	For every admissible family \(\mathcal K\), the Kendall profile
	\(q_\tau\) is smooth, strictly decreasing and locally
	bi-Lipschitz on \(\mathbb R\). Hence the conclusions of
	Theorem~\ref{thm:informative-functionals}\textup{(iii)--(iv)}
	apply.
	
	The Spearman profile \(q_{\rho_S}\) is continuous and strictly
	decreasing on \(\mathbb R\). Hence
	Theorem~\ref{thm:informative-functionals}\textup{(iii)} applies:
	the matrix of pairwise Spearman coefficients determines the
	copula, its reduced hierarchy and its node parameters.
	
	The same conclusions as in
	Theorem~\ref{thm:informative-functionals}\textup{(iii)--(iv)}
	hold for lower-tail dependence in the Clayton family and for
	upper-tail dependence in the Gumbel and Joe families.
	All remaining tail profiles in
	Table~\ref{tab:tail-profiles} are identically zero.
\end{corollary}

\begin{proof}
	The Kendall assertion follows immediately from \eqref{eq:universal-kendall-profile}.
	
	We next consider Spearman's rho. The five families considered here
	are known to be positively ordered by their parameter in the
	concordance order, that is, if \(\theta_1\le\theta_2\), then
	\(
	C^{(2)}_{\theta_1}(u,v)
	\le
	C^{(2)}_{\theta_2}(u,v).
	\) For the corresponding family-specific
	statements, see
	\cite[Exercise~2.32, Examples~4.12 and~4.14,
	Table~4.1 and Exercise~4.18(a)]{Nelsen2006}.
	For completeness, we give a short direct verification in Appendix \ref{app:concavity}.

	If \(\theta_1<\theta_2\), equality cannot hold identically by
	Assumption~\ref{ass:bivariate-identifiability}. Since the copulas are continuous, the
	inequality is therefore strict on a set of positive Lebesgue
	measure. Consequently
	\[
	\bar\rho_{\mathcal K}(\theta)
	:=
	12\int_0^1\!\!\int_0^1
	C^{(2)}_\theta(u,v)\,du\,dv-3
	\]
	is strictly increasing in \(\theta\). Moreover,
	Assumption~\ref{it:continuous-family} implies that
	\(\bar\rho_{\mathcal K}\) is continuous. By \eqref{eq:kendall-height},
	\(h_{\mathcal K}:\Theta\to\mathbb R\) is a strictly decreasing
	homeomorphism, and therefore its inverse
	\(\theta_{\mathcal K}\) is continuous and strictly decreasing.
	Hence
	\(
	q_{\rho_S}
	=
	\bar\rho_{\mathcal K}\circ\theta_{\mathcal K}
	\)
	is continuous and strictly decreasing on \(\mathbb R\).
	Theorem~\ref{thm:informative-functionals}\textup{(iii)} therefore
	applies to Spearman's rho.
	
	The tail statements follow from the standard family-specific
	tail-dependence formulas displayed in
	Table~\ref{tab:tail-profiles}. After the substitution
	\(\theta=\theta_{\mathcal K}(h)\), the nonconstant profiles are
	explicitly smooth with nowhere-vanishing derivative and hence
	locally bi-Lipschitz. Thus
	Theorem~\ref{thm:informative-functionals}\textup{(iii)--(iv)}
	applies to lower-tail dependence for Clayton and to upper-tail
	dependence for Gumbel and Joe. The remaining tail profiles are
	identically zero.
\end{proof}

%
%

\FloatBarrier

	\section{Statistical recovery from pairwise dependence summaries}
	\label{sec:statistical-recovery}
	
	We now explain how the preceding geometric construction leads to natural
	statistical estimators of the reduced homogeneous nested Archimedean copula.
	The main point is that informative pairwise copula summaries are not merely
	diagnostic quantities: on the reduced model class they are coordinates, up to
	a one-dimensional transformation. This allows one to estimate the cophenetic
	coordinates, project them onto the stratified copula space, and recover the
	reduced hierarchy.
	
	Throughout this section, let
	\(
	X_1,\ldots,X_n
	\)
	be an i.i.d. sample from a \(d\)-dimensional distribution with continuous
	margins and copula
	\(
	C_0\in\mathcal C_d^{\mathcal K}.
	\) Denote \(H_0:=\Psi_{\cl K}(C_0).\)
	Let \(T:\mathfrak C_2\to\mathbb R\) be a bivariate copula functional, and let
	\(
	\mathsf T(C)
	=
	(T(C_{ij}))_{1\le i<j\le d}
	\)
	be the associated pairwise functional matrix. We assume that the profile
	\(
	q_T:\R\to q_T(\R)
	\)
	is a homeomorphism.
	Let
	\[
	\widehat{\mathsf T}_n
	=
	\bigl(\widehat T_{ij,n}\bigr)_{1\le i<j\le d}
	\]
	be an estimator of the pairwise \(T\)-matrix. We shall use the following assumption.
	
	\begin{assumption}{S}
		\label{ass:pairwise-summary-consistency}
		The estimator \(\widehat{\mathsf T}_n\) satisfies
		\(
		\|
		\widehat{\mathsf T}_n-\mathsf T(C_0)
		\|_2
		\to 0
		\)
		in probability, and,
		\(
		\mathbb P(
		\widehat{\mathsf T}_n\in q_T(\R)^{N}
		)
		\to 1.
		\)
	\end{assumption}

	The second condition only ensures that \(q_T^{-1}\) can be applied
	componentwise with probability tending to one. In applications, one may also
	replace \(\widehat{\mathsf T}_n\) by a truncated version taking values in
	\(q_T(\R)^{N}\); this does not change the conclusions below whenever
	\(\mathsf T(C_0)\) lies in the interior of \(q_T(\R)^{N}\).	
	Assumption \assref{ass:pairwise-summary-consistency} will be in force,
	together with Assumptions \assref{ass:K} and
	\assref{ass:kendall-height}, throughout this section.
	
	On the event
	\(
	\widehat{\mathsf T}_n\in q_T(\R)^{N},
	\)
	define the empirical cophenetic vector
	\(
	\widehat H_n
	:=
	Q_T^{-1}(\widehat{\mathsf T}_n).
	\)
	Outside this event, define
	\(\widehat H_n\) arbitrarily. Denote \(H_0:=\Psi_{\mathcal K}(C_0).\)	
	For \(y \in \R^N\) and a non-empty closed set \(F\subset\R^N\), write
	\(
	\Pi_F(y)
	:=
	\operatorname*{argmin}_{u\in F}\|y-u\|_2.
	\)
	This set is non-empty for every \(y\in\R^N\). If it is a singleton, we
	denote its unique element by \(\pi_F(y)\).
	Let
	\[
	\widehat H_n^\Pi
	\in
	\Pi_{\mathfrak U_d}
	(
	\widehat H_n).
	\]
	Since \(\mathfrak U_d\) is a finite union of closed convex
	polyhedral cones, it admits a Borel measurable nearest-point
	selection. Indeed, enumerate its maximal cones as
	\(\overline{\mathfrak U}^{\sigma_1},\ldots,
	\overline{\mathfrak U}^{\sigma_M}\), let \(\Pi_k\) denote the
	orthogonal projection onto
	\(\overline{\mathfrak U}^{\sigma_k}\), and, in case several cones
	give the same minimal distance, select the one with the smallest
	index. Each \(\Pi_k\) is continuous, so this rule defines a Borel
	measurable map from \(\R^N\) into \(\mathfrak U_d\).
	We then define the projected copula
	estimator by
	\[
	\widehat C_n^\Pi
	:=R_{\mathcal K}(\widehat H_n^\Pi).
	\]
	
	\begin{theorem}
		\label{thm:projected-estimator-consistency}
		\(
		d_{\mathcal C}(\widehat C_n^\Pi,C_0)
		\to 0
		\)
		in probability.
		Moreover, if \(Q_T^{-1}:q_T(\R)^N \to \R^N\) is locally Lipschitz in a neighbourhood of \(\mathsf T(C_0)\) and
		\(
		\|
		\widehat{\mathsf T}_n-\mathsf T(C_0)
		\|_2
		=
		O_{\mathbb P}(\varepsilon_n)
		\)
		for some deterministic sequence \(\varepsilon_n\downarrow0\), then
		\[
		d_{\mathcal C}(\widehat C_n^\Pi,C_0)
		=
		O_{\mathbb P}(\varepsilon_n).
		\]
	\end{theorem}
	
	\begin{proof}
		Since \(q_T:\R\to q_T(\R)\) is a homeomorphism, the map
		\(
		Q_T^{-1}:q_T(\R)^{N}\to \R^{N}
		\)
		is continuous. Therefore, on the event
		\(\{\widehat{\mathsf T}_n\in q_T(\R)^{N}\}\),
		\(
		\widehat{\mathsf T}_n\to \mathsf T(C_0)\) implies \(
		\widehat H_n=Q_T^{-1}(\widehat{\mathsf T}_n)
		\to
		Q_T^{-1}(\mathsf T(C_0))
		=
		H_0.
		\)
		Since \(\{\widehat{\mathsf T}_n\in q_T(\R)^{N}\}\) has probability tending to one, we have
		\begin{equation}
			\label{eq:cophenetic-consistency}
			\widehat H_n
			\to
			H_0 \quad \text{in probability in }\R^N.
		\end{equation}
		Since \(H_0\in\mathfrak U_d\), the definition of
		\(\widehat H_n^\Pi\) gives
		\(
		\|\widehat H_n-\widehat H_n^\Pi\|_2
		\le
		\|\widehat H_n-H_0\|_2.
		\)
		Therefore, by \eqref{eq:cophenetic-consistency},
		\[
		\begin{aligned}
			d_{\mathcal C}(\widehat C_n^\Pi,C_0)=\|\widehat H_n^\Pi-H_0\|_2
			&\le
			\|\widehat H_n^\Pi-\widehat H_n\|_2
			+
			\|\widehat H_n-H_0\|_2  \le
			2\|\widehat H_n-H_0\|_2\to0 \quad \text{in probability.}
		\end{aligned}
		\]

		If in addition
		\(Q_T^{-1}\) is locally Lipschitz on a neighborhood of \(\mathsf T(C_0)\) we have
		\[
		d_{\mathcal C}(\widehat C_n^\Pi,C_0)\le 2\|\widehat H_n-H_0\|_2
		=
		2\left\|
		Q_T^{-1}(\widehat{\mathsf T}_n)
		-
		Q_T^{-1}(\mathsf T(C_0))
		\right\|_2
		=
		O_{\mathbb P}(\varepsilon_n).
		\]
		This concludes the proof.
	\end{proof}
	
	We observe that if \(q_T\) is locally bi-Lipschitz, as in the cases discussed in Section \ref{subsec:pairwise-functionals}, the \(Q_T^{-1}\) Lipschitzianity assumption is satisfied. 
	
	We next discuss recovery of the reduced hierarchy. For
	\(H\in\mathfrak U_d\) and \(\varepsilon\ge0\), define
	\(\ctr_\varepsilon(H)\) as the hierarchy obtained from \(\sigma(H)\) by
	contracting all non-root internal edges \(v\) such that
	\(\eta_v(H)\le\varepsilon\). In particular,
	\(\ctr_0(H)=\sigma(H)\).	
	For \(C\in\mathcal C_d^{\mathcal K}\), define
	\[
	\sigma_\varepsilon(C)
	:=\ctr_\varepsilon(\Psi_{\mathcal K}(C)).
	\]
	Thus \(\sigma_0(C)\) is the reduced hierarchy of \(C\).	
	The following lemma is the local stability statement behind hierarchy
	recovery. Near a point in a given fan stratum, only refinements of its reduced
	hierarchy can occur; the extra refinement edges have small increments, while
	the genuine edges remain bounded away from zero.
	
	\begin{lemma}
		\label{lem:threshold-local-stability}
		Let \(\sigma(H_0)=\bar \sigma\). There exist constants
		\(a,b>0\) and a neighborhood \(V\) of \(H_0\) in
		\(\mathfrak U_d\) such that, if \(H\in V\),
		\begin{equation}
			\label{eq:lem2-hypothesis}
		\|H-H_0\|_2\le a\varepsilon,
		\qquad 0<\varepsilon<b,
		\end{equation}
		then \(\ctr_\varepsilon(H)=\bar \sigma\).
	\end{lemma}
	
	\begin{proof}
		By \eqref{eq:local-refinement-union} and Theorem~\ref{thm:ultrametric-polyhedral-characterization}, there exists a neighborhood \(V\)
		of \(H_0\) such that every \(H\in V\) has reduced hierarchy
		\(\tau=\sigma(H)\) such that \(
		\bar\sigma\preceq\tau.
		\)
		Fix such a refinement \(\tau\). Since \(L_\tau:\R^{K(\tau)}
		\to
		V_\tau\) is bijective and linear,
		\(
		L_\tau^{-1}
		:
		V_\tau
		\to
		\R^{K(\tau)}
		\)
		is linear and, thus, continuous. For
		\(G\in V_\tau\),
		write
		\(
		L_\tau^{-1}(G)
		=(r^\tau(G),
		(\eta_v^\tau(G))_{v\in I_\tau}).
		\)
		There is therefore a constant \(C_\tau<\infty\) such that
		\(
		|\eta_v^\tau(G)-\eta_v^\tau(G')|
		\le
		C_\tau\|G-G'\|_2,
		\)
		for all
		\(G,G'\in V_\tau\) and \(v \in I_\tau\).
		Since there are only finitely many refinements of \(\bar\sigma\), we may set
		\(
		C
		:=
		\max_{\tau:\,\bar\sigma\preceq\tau} C_\tau
		<\infty.
		\) Moreover, suppose \(I_{\bar\sigma}\neq\emptyset\) and set \(m
		:=
		\min_{v\in I_{\bar\sigma}}
		\eta_v(H_0)
		>0.\)
	
		If \(H\in V\cap\mathfrak U^\tau\), we have
		\begin{gather*}
			0<\eta_v(H)
			=
			\eta_v^\tau(H)
			\le
			C\|H-H_0\|_2,
			\qquad
			A_v\in
			\mathscr A(\tau)\setminus\mathscr A(\bar\sigma),\\
			\eta_v(H)
			=
			\eta_v^\tau(H)
			\ge
			\eta_v^\tau(H_0)-C\|H-H_0\|_2
			\ge
			m-C\|H-H_0\|_2\qquad A_v \in \mathscr{A}(\bar \sigma)\setminus\{[d]\}.
		\end{gather*}
		Choose
		\(a>0\) such that \(
		Ca<1/2\) and \(
		b:=m/2.
		\)
		If \eqref{eq:lem2-hypothesis} holds,
		then 
		\begin{gather*}
			\eta_v(H)
			\le
			Ca\varepsilon
			<
			\varepsilon,
			\qquad
			A_v\in
			\mathscr A(\tau)\setminus\mathscr A(\bar\sigma),\\
			\eta_v(H)
			\ge
			m-Ca\varepsilon
			>
			m-\frac{\varepsilon}{2}
			>
			\varepsilon,\qquad A_v \in \mathscr{A}(\bar \sigma)\setminus\{[d]\}.
		\end{gather*}
		Thus \(\ctr_\varepsilon(H)\) contracts exactly the edges corresponding to
		\(\mathscr A(\tau)\setminus\mathscr A(\bar\sigma)\) and preserves exactly
		the non-root internal clusters of \(\bar\sigma\). Consequently,
		\(
		\ctr_\varepsilon(H)=\bar\sigma.
		\)
		
		If \(I_{\bar\sigma}=\varnothing\), there are no genuine non-root internal
		edges to preserve. Choose any \(a>0\) such that \(Ca<1\) and any \(b>0\).
		Then, if \(H\in V\cap\mathfrak U^\tau\), we have
		\(
		\eta_v(H)
		\le
		Ca\varepsilon
		<
		\varepsilon
		\)
		for every non-root internal edge of \(\sigma(H)\). Hence all such edges are
		contracted and again
		\(\ctr_\varepsilon(H)=\bar\sigma\).
	\end{proof}

	\begin{theorem}
		\label{thm:thresholded-hierarchy-recovery}
		Let \(\sigma(H_0)=\bar \sigma\). Let \(\widehat C_n\) be any
		\(\mathcal C_d^{\mathcal K}\)-valued estimator such that
		\(d_{\mathcal C}(\widehat C_n,C_0)\to0\) in probability.
		
		\begin{enumerate}[label=(\roman*), ref=(\roman*)]
			\item If \(\varepsilon_n\downarrow0\) and
			\(
			d_{\mathcal C}(\widehat C_n,C_0)
			=o_{\mathbb P}(\varepsilon_n),
			\)
			then
			\(
			\mathbb P(
			\sigma_{\varepsilon_n}(\widehat C_n)=\bar\sigma
			)\to1.
			\)
			
			\item If \(\bar\sigma\in\mathcal B_d\), then
			\(
			\mathbb P(
			\sigma_0(\widehat C_n)=\bar\sigma
			)\to1.
			\)
		\end{enumerate}
	\end{theorem}

	\begin{proof}
		Set \(\widehat G_n:=\Psi_{\mathcal K}(\widehat C_n)\). By definition of
		the cophenetic metric,
		\(
		\|\widehat G_n-H_0\|_2
		=d_{\mathcal C}(\widehat C_n,C_0).
		\)
		The first assertion follows directly from
		Lemma~\ref{lem:threshold-local-stability}.
		
		If \(\bar \sigma\) is binary, \(\mathfrak U^{\bar \sigma}\) is the relative
		interior of a maximal cone. Since the fan is finite, a sufficiently small
		neighborhood of \(H_0\) in its support is contained in
		\(\mathfrak U^{\bar \sigma}\). Hence
		\(\mathbb P(\widehat G_n\in\mathfrak U^{\bar \sigma})\to1\), which is
		equivalent to~(ii).
	\end{proof}

	We finally observe that Theorem~\ref{thm:thresholded-hierarchy-recovery} applies 
	to the projected estimator \(\widehat C_n^\Pi\) using Theorem~\ref{thm:projected-estimator-consistency}. In particular, suppose that \(Q_T^{-1}\) is locally Lipschitz in a
	neighborhood of \(\mathsf T(C_0)\), and that, for some
	deterministic sequence \(r_n\downarrow0\),
	\(\|\widehat{\mathsf T}_n-\mathsf T(C_0)\|_2
	=
	O_{\mathbb P}(r_n).
	\)
	Then, by Theorem~\ref{thm:projected-estimator-consistency},
	\(
	d_{\mathcal C}(\widehat C_n^\Pi,C_0)
	=
	O_{\mathbb P}(r_n),
	\)
	and hence, for every deterministic sequence \(\varepsilon_n\downarrow0\)
	such that \(r_n=o(\varepsilon_n)\), Theorem~\ref{thm:thresholded-hierarchy-recovery} gives
	\(
	\mathbb P(
	\sigma_{\varepsilon_n}(\widehat C_n^\Pi)=\bar\sigma)\to1.
	\)
	Thus, at a multifurcating hierarchy, consistent recovery is obtained by
	letting the contraction threshold vanish more slowly than the estimation
	error.
	
	\section{Singular asymptotics on the ultrametric fan}
	\label{sec:singular-clt}
	
	The asymptotic distribution of the projected estimator reflects
	the general tangent-cone principle for constrained estimation
	developed by \citet{Geyer1994}: local quadratic behavior of the
	criterion and the tangent geometry of the feasible set determine
	the limiting distribution. Since \cite[Theorem~4.4]{Geyer1994} is formulated for empirical-average
	criteria under specific regularity assumptions, here we establish the limit directly,
	using the exact local rescaling identity of
	Lemma~\ref{lem:local-tangent-fan} and the nearest-point property
	of the estimator. When the true copula is multifurcating,
	several refinement strata meet at the true point, and the
	limiting feasible set is a generally non-convex tangent fan.
	The limiting distribution is the Euclidean projection of a
	Gaussian vector onto this fan; almost-sure uniqueness of that
	projection follows from \citet[Proposition~4.2]{Geyer1994}. We now state the relevant assumptions.
	
\begin{assumption}{P}
	\label{ass:projection}
	For every \(1\le i<j\le d\), the profile \(q_T\) is differentiable in a neighbourhood of
	\(H_{0,ij}\), with
	\(
	q_T'(H_{0,ij})\neq0.
	\)
	Set
	\(
	D_T
	:=
	\operatorname{diag}
	\left(
	q_T'(H_{0,ij})
	\right)_{1\le i<j\le d}.
	\)
	The dependence-summary estimator satisfies
	\[
	\sqrt n
	\left(
	\widehat{\mathsf T}_n-\mathsf T(C_0)
	\right)
	\Rightarrow
	Z,
	\qquad
	Z\sim\mathcal N_N(0,\Sigma_T).
	\]
\end{assumption}

\begin{assumption}{D}
	\label{ass:density}
	The Gaussian vector
	\(
	D_T^{-1}Z
	\)
	has a density with respect to Lebesgue measure on \(\mathbb R^N\).
\end{assumption}

\begin{remark}
	\label{rem:lebesgue-density}
	Since \(D_T\) is invertible and \(Z\) is Gaussian, Assumption~\assref{ass:density} is equivalent to
	positive definiteness of \(\Sigma_T\).
	
	Moreover, Assumption~\assref{ass:projection} implies
	Assumption~\assref{ass:pairwise-summary-consistency}: the central
	limit theorem gives
	\(
	\widehat{\mathsf T}_n-\mathsf T(C_0)
	\to0
	\)
	in probability, while \(q_T(\mathbb R)^N\) is open and contains
	\(\mathsf T(C_0)\).
\end{remark}

	
%

	\begin{theorem}
		\label{thm:whole-space-singular-clt}
		Let  \(\bar\sigma:=\sigma(H_0)\).
		Under Assumptions \assref{ass:projection} and \assref{ass:density}, \(\widehat H_n^\Pi\) satisfies
	\begin{equation}
		\label{eq:whole-space-singular-clt}
		\sqrt n\left(
		\widehat H_n^\Pi-H_0
		\right)
		\Rightarrow
		\pi_{\mathcal T_{\bar\sigma}}(D_T^{-1}Z).
	\end{equation}
		Moreover, 
		\begin{equation}
			\label{eq:asymptotic-binary-refinement}
			\mathbb P\left(
			\bar\sigma\preceq\sigma(\widehat H_n^\Pi)\in\mathcal B_d
			\right)
			\longrightarrow1.
		\end{equation}
		In particular,
		\[
		\mathbb P(\sigma(\widehat H_n^\Pi)=\bar\sigma)
		\longrightarrow
		\begin{cases}
			1, & \bar\sigma\in\mathcal B_d,\\
			0, & \bar\sigma\notin\mathcal B_d.
		\end{cases}
		\]
	\end{theorem}

	\begin{proof}
		Since \(D_T\) is invertible, the multivariate delta method applied to
		the componentwise inverse \(Q_T^{-1}\) gives
		\begin{equation*}
			Y_n
			:=
			\sqrt n(\widehat H_n-H_0)
			=
			D_T^{-1}
			\sqrt n\left(
			\widehat{\mathsf T}_n-\mathsf T(C_0)
			\right)
			+o_{\mathbb P}(1)
			\Rightarrow Y:=D_T^{-1}Z.
		\end{equation*}
		The arbitrary definition of \(\widehat H_n\) outside the event
		\(\widehat{\mathsf T}_n\in q_T(\R)^N\) does not affect the limit because
		that event has probability tending to one.
		
		Set
		\(
		U_n
		:=
		\sqrt n(\widehat H_n^\Pi-H_0)
		\)
		and
		\(
		\mathfrak W_n
		:=
		\sqrt n(\mathfrak U_d-H_0).
		\)
		By the definition of \(\widehat H_n^\Pi\),
		\begin{equation}
			\label{eq:scaled-projection-problem}
			U_n\in\Pi_{\mathfrak W_n}(Y_n).
		\end{equation}
		Moreover, \(0\in\mathfrak W_n\), and hence
		\(
		\|U_n-Y_n\|_2
		\le
		\|Y_n\|_2,
		\)
		from which it follows that
		\begin{equation}
			\label{eq:Y-dominates-U}
			\|U_n\|_2
			\le
			2\|Y_n\|_2.
		\end{equation}
		Thus, tightness of \((Y_n)\) gives joint tightness of \((Y_n,U_n)\).
		
		Fix \(M<\infty\). On the event \(\{\|Y_n\|_2\le M\}\),
		\eqref{eq:Y-dominates-U} implies that every minimizer in
		\eqref{eq:scaled-projection-problem} belongs to
		\(\overline B_{2M}(0)\).
		The same is true for every minimizer of the distance from \(Y_n\) to
		\(\mathcal T_{\bar\sigma}\), because
		\(0\in\mathcal T_{\bar\sigma}\).
		By Lemma~\ref{lem:local-tangent-fan}, for all sufficiently large \(n\),
		\(
		\mathfrak W_n\cap\overline B_{2M}(0)
		=
		\mathcal T_{\bar\sigma}\cap\overline B_{2M}(0).
		\)
		Consequently,
		\begin{equation}
			\label{eq:local-projection-identification}
			U_n\in\Pi_{\mathcal T_{\bar\sigma}}(Y_n)
		\end{equation}
		whenever \(\|Y_n\|_2\le M\) and \(n\) is sufficiently large.
		
		Let \((n_k)\) be an arbitrary subsequence. By tightness, it has a further
		subsequence, not relabelled, such that
		\(
		(Y_{n_k},U_{n_k})\Rightarrow(Y,U)
		\)
		for some random vector \(U\). By the Skorokhod representation theorem, we may realize this
		convergence almost surely on a common probability space. Now \(Y_{n_k}\to Y\) almost surely implies that, for almost every
		outcome, the sequence \((Y_{n_k})\) is bounded, while \eqref{eq:scaled-projection-problem} holds almost surely for \(U_{n_k}\) after the Skorokhod representation because it is
		a property of the joint distribution. Therefore
		\eqref{eq:local-projection-identification} holds eventually. 
		Since \(\mathcal T_{\bar\sigma}\) is closed, \(U \in \mathcal T_{\bar\sigma}\). Moreover, passing to the limit in
		the inequalities
		\[
		\|U_{n_k}-Y_{n_k}\|_2
		\le \|w-Y_{n_k}\|_2,
		\qquad w\in\mathcal T_{\bar\sigma},
		\]
		gives
		\(U\in\Pi_{\mathcal T_{\bar\sigma}}(Y)\) almost surely.
		The graph of
		the metric-projection correspondence onto a closed set is closed, so
		\(
		U\in\Pi_{\mathcal T_{\bar\sigma}}(Y)
		\)
		almost surely. By Assumption~\assref{ass:density}, \(Y\) has a density with
		respect to Lebesgue measure. Since
		\(\mathcal T_{\bar\sigma}\) is closed,
		\citet[Proposition~4.2]{Geyer1994} implies that its nearest
		point to \(Y\) is unique almost surely. Thus
		\(
		U=\pi_{\mathcal T_{\bar\sigma}}(Y)\) almost surely.
		Every subsequence therefore has a further subsequence with
		this same weak limit, proving
		\eqref{eq:whole-space-singular-clt}.

		It remains to prove \eqref{eq:asymptotic-binary-refinement}.
		Set \(\cl G_{\bar\sigma}=\{\gamma \in \cl H_d : \bar\sigma \preceq \gamma ,\gamma \notin \cl B_d\}\) and
		\(
		\mathcal S_{\bar\sigma}
		:=
		\bigcup_{\gamma \in \cl G_{\bar\sigma}}
		\mathcal T_{\gamma\mid\bar\sigma}.
		\)
		This is the union of the non-maximal cones of
		\(\mathcal T_{\bar\sigma}\), and is therefore a closed subset of the
		finite fan \(\mathcal T_{\bar\sigma}\).
		We first show that
		\begin{equation}
			\label{eq:limit-avoids-skeleton}
			\mathbb P\left(
			\pi_{\mathcal T_{\bar\sigma}}(Y)
			\in\mathcal S_{\bar\sigma}
			\right)
			=0.
		\end{equation}
		Suppose that \(y\in\R^N\) has a unique nearest point
		\(
		u:=\pi_{\mathcal T_{\bar\sigma}}(y)\in\mathcal S_{\bar\sigma}.
		\)
		There is a unique
		\(\gamma\in\mathcal G_{\bar\sigma}\) such that
		\(
		u\in
		\operatorname{relint}
		\bigl(\mathcal T_{\gamma\mid\bar\sigma}\bigr).
		\)
		Since
		\(
		\Span(\mathcal T_{\gamma\mid\bar\sigma})
		=
		V_\gamma,
		\)
		for every \(w\in V_\gamma\), one has
		\(
		u\pm tw\in\mathcal T_{\gamma\mid\bar\sigma}
		\)
		for all sufficiently small \(t>0\). Write \(r=y-u.\) Since \(u\) minimizes
		\(z\mapsto\|y-z\|_2^2\) over \(\mathcal T_{\bar\sigma}\), it follows that
		\(
		0
		=
		\frac{d}{dt}|_{t=0}
		\frac12\|y-(u+tw)\|_2^2
		=
		-\langle r,w\rangle .
		\)
		Hence
		\begin{equation}
			\label{eq:residual-orthogonal-face}
			r\perp V_\gamma.
		\end{equation}
		
		Let \(v\) be a multifurcating vertex of \(\gamma\), with child
		leaf-sets
		\(
		B_1,\ldots,B_m,
		\ m\ge3.
		\)
		Fix \(p<q\), and let \(\tau_{pq}\) be the hierarchy obtained from
		\(\gamma\) by adding the single cluster \(B_p\cup B_q\).
		Then \(\gamma\prec\tau_{pq}\), and
		\(
		\mathcal T_{\tau_{pq}\mid\bar\sigma}
		=
		\mathcal T_{\gamma\mid\bar\sigma}
		+
		\operatorname{cone}\{-b_{B_p\cup B_q}\}.
		\)
		Hence, since
		\(u\in\mathcal T_{\gamma\mid\bar\sigma}\), one has
		\(
		u-tb_{B_p\cup B_q}
		\in
		\mathcal T_{\tau_{pq}\mid\bar\sigma}
		\subset
		\mathcal T_{\bar\sigma},\) for \(t\ge0.
		\)
		Since \(u\) minimizes \(z\mapsto\|y-z\|_2^2\) over
		\(\mathcal T_{\bar\sigma}\), for every \(t>0\),
		\[
		\begin{aligned}
			0
			&\le
			\frac12
			\bigl\|y-(u-tb_{B_p\cup B_q})\bigr\|_2^2
			-
			\frac12\|y-u\|_2^2=
			-t\langle r,-b_{B_p\cup B_q}\rangle
			+
			\frac{t^2}{2}\|b_{B_p\cup B_q}\|_2^2 .
		\end{aligned}
		\]
		Dividing by \(t\) and letting \(t\downarrow0\) gives
		\begin{equation}
			\label{eq:one-sided-refinement-optimality}
			\langle r,-b_{B_p\cup B_q}\rangle\le0.
		\end{equation}
		
		On the other hand,
		\begin{equation}
			\label{eq:multifurcation-generator-identity}
			\sum_{p<q} b_{B_p\cup B_q}
			=
			b_{A_v}
			+
			(m-2)\sum_{p=1}^m b_{B_p},
		\end{equation}
		where
		\(A_v=B_1\sqcup\cdots\sqcup B_m\), and \(b_{B_p}\) is interpreted as
		zero when \(B_p\) is a singleton. To verify
		\eqref{eq:multifurcation-generator-identity}, consider a coordinate
		indexed by two leaves. If the leaves lie in distinct child sets
		\(B_p\) and \(B_q\), they occur together in exactly one union
		\(B_a\cup B_b\), and both sides equal \(1\). If they lie in the same
		child set \(B_p\), they occur together in exactly \(m-1\) such unions,
		while the right-hand side equals \(1+(m-2)=m-1\). All remaining
		coordinates vanish on both sides.
		
		Moreover, \(A_v\) and every non-singleton \(B_p\) are internal clusters
		of \(\gamma\), so the right-hand side of
		\eqref{eq:multifurcation-generator-identity} belongs to \(V_\gamma\).
		Therefore, by \eqref{eq:residual-orthogonal-face},
		\[
		0
		=
		\left\langle
		r,\sum_{p<q}-b_{B_p\cup B_q}
		\right\rangle
		=
		\sum_{p<q}
		\langle r,-b_{B_p\cup B_q}\rangle.
		\]
		Every summand is non-positive by
		\eqref{eq:one-sided-refinement-optimality}, and hence
		\begin{equation}
			\label{eq:refinement-residual-orthogonality}
			\langle r,-b_{B_p\cup B_q}\rangle=0,
			\qquad p<q.
		\end{equation}
		
		Now fix one pair \(p<q\) and write \(\tau:=\tau_{pq}\). Since
		\(\tau\) is a genuine one-edge refinement of \(\gamma\), we have
		\(
		V_\tau
		=
		V_\gamma+\Span\{-b_{B_p\cup B_q}\}\) and \(
		\dim V_\tau=\dim V_\gamma+1.
		\)
		Equations~\eqref{eq:residual-orthogonal-face} and
		\eqref{eq:refinement-residual-orthogonality} therefore give
		\(
		r\perp V_\tau.
		\)
		Since
		\(u\in\mathcal T_{\gamma\mid\bar\sigma}\subset V_\gamma\),
		we conclude that
		\(
		y=u+r
		\in
		V_\gamma+V_\tau^\perp.
		\)
		Because \(V_\gamma\subset V_\tau\) and
		\(\dim V_\tau=\dim V_\gamma+1\),
		\(
		\dim\bigl(V_\gamma+V_\tau^\perp\bigr)
		=
		\dim V_\gamma+N-\dim V_\tau
		=
		N-1,
		\)
		and, thus, \(V_\gamma+V_\tau^\perp\) is a proper linear subspace of
		\(\R^N\).
		
		For each \(\gamma\in\mathcal G_{\bar\sigma}\), choose one such
		one-edge refinement \(\tau_\gamma\). We have shown that
		\[
		\left\{
		y\in\R^N:
		\pi_{\mathcal T_{\bar\sigma}}(y)
		\in
		\operatorname{relint}
		(\mathcal T_{\gamma\mid\bar\sigma})
		\right\}
		\subset
		V_\gamma+V_{\tau_\gamma}^\perp
		\]
		at every point where the nearest projection is unique. Since
		\(\mathcal G_{\bar\sigma}\) is finite, the set
		\(
		\{y:
		\pi_{\mathcal T_{\bar\sigma}}(y)
		\in\mathcal S_{\bar\sigma}\}
		\)
		is, up to the Lebesgue-null set where the nearest projection is not
		unique, contained in a finite union of proper linear subspaces of
		\(\R^N\). It therefore has Lebesgue measure zero. Since \(Y\) has a
		Lebesgue density, \eqref{eq:limit-avoids-skeleton} follows.

		We finally transfer this property to the finite-sample hierarchy. Set
		\(
		U_n:=\sqrt n(\widehat H_n^\Pi-H_0).
		\)
		By \eqref{eq:whole-space-singular-clt},
		\(
		U_n
		\Rightarrow
		\pi_{\mathcal T_{\bar\sigma}}(Y),
		\)
		and hence \((U_n)\) is tight.
		Moreover, by Theorem~\ref{thm:projected-estimator-consistency} and
		\eqref{eq:local-refinement-union},
		\begin{equation}
			\label{eq:estimated-hierarchy-refines-truth}
			\mathbb P\left(
			\bar\sigma\preceq\sigma(\widehat H_n^\Pi)
			\right)
			\longrightarrow1.
		\end{equation}
		Fix \(R<\infty\). By \eqref{eq:single-refinement-exact-local-rescaling}, for all sufficiently large
		\(n\), on the event
		\(
		E_{n,R}:=\{
		\|U_n\|_2\le R,\ 
		\bar\sigma\preceq\sigma(\widehat H_n^\Pi),\
		\sigma(\widehat H_n^\Pi)\notin\mathcal B_d\},
		\)
		we have
		\(
		U_n
		\in
		\mathcal T_{\sigma(\widehat H_n^\Pi)\mid\bar\sigma}
		\subset
		\mathcal S_{\bar\sigma}.
		\)
		Since \(\mathcal S_{\bar\sigma}\) is closed,
		\eqref{eq:whole-space-singular-clt},
		\eqref{eq:limit-avoids-skeleton}, and the Portmanteau theorem give
		\[
		\limsup_{n\to\infty}\P(E_{n,R})\le\limsup_{n\to\infty}
		\mathbb P(U_n\in\mathcal S_{\bar\sigma})
		\le
		\mathbb P\left(
		\pi_{\mathcal T_{\bar\sigma}}(Y)
		\in\mathcal S_{\bar\sigma}
		\right)
		=0.
		\]
		Letting \(R\to\infty\) and using
		\eqref{eq:estimated-hierarchy-refines-truth} yields \eqref{eq:asymptotic-binary-refinement}.
		
		If \(\bar\sigma\) is binary, it has no strict refinement, so the preceding
		display implies
		\(
		\mathbb P(\sigma(\widehat H_n^\Pi)=\bar\sigma)\to1.
		\)
		If \(\bar\sigma\) is not binary, the event
		\(
		\{\bar\sigma\preceq\sigma(\widehat H_n^\Pi)\in\mathcal B_d\}
		\)
		implies
		\(
		\sigma(\widehat H_n^\Pi)\ne\bar\sigma.
		\)
		Hence
		\(
		\mathbb P(\sigma(\widehat H_n^\Pi)=\bar\sigma)\to0,
		\)
		and the proof is complete.
	\end{proof}

	The form of the limit is related to the constrained asymptotics for
	overparameterized hierarchical Archimedean copulas studied by
	\citet{PerreaultTangPanReid2025}, where equality of adjacent parameters
	places the true model on the boundary of a fixed hierarchical parameter
	space. Here the hierarchy itself is not fixed: locally around a
	multifurcating hierarchy, all admissible refinements are simultaneously
	feasible. The relevant local set is therefore the generally nonconvex
	tangent fan \(\mathcal T_{\bar\sigma}\), rather than a single constraint
	cone, and the limiting projection describes both continuous parameter
	fluctuations and the refinement selected at first order.
	
	Indeed, the cones of \(\mathcal T_{\bar\sigma}\) are indexed by the
	refinements of \(\bar\sigma\) and meet along common faces. If \(
	\pi_{\mathcal T_{\bar\sigma}}(D_T^{-1}Z)
	\in
	\operatorname{relint}
	\bigl(\mathcal T_{\tau\mid\bar\sigma}\bigr),\)	
	we call \(\tau\) the \emph{first-order hierarchy} selected by the
	projected estimator. The additional clusters in
	\(\mathscr A(\tau)\setminus\mathscr A(\bar\sigma)\) then correspond to
	increments of \(\widehat H_n^\Pi\) of order \(n^{-1/2}\).
	By the second assertion of
	Theorem~\ref{thm:whole-space-singular-clt}, the limiting projection lies almost surely
	in the relative interior of a maximal cone of
	\(\mathcal T_{\bar\sigma}\). Hence the first-order hierarchy is almost
	surely a binary refinement of \(\bar\sigma\).
	
	This also explains the non-Gaussian limit at a multifurcating hierarchy.
	Different realizations of \(D_T^{-1}Z\) may project onto different maximal
	cones of \(\mathcal T_{\bar\sigma}\), so the map
	\(y\mapsto\pi_{\mathcal T_{\bar\sigma}}(y)\) is only piecewise linear.
	The limit therefore records simultaneously the fluctuations of the
	genuine parameters and the random binary resolution of the
	multifurcations.
	
	Theorem~\ref{thm:whole-space-singular-clt} also identifies the appropriate
	threshold scale in
	Theorem~\ref{thm:thresholded-hierarchy-recovery}. Since \(
	d_{\mathcal C}(\widehat C_n^\Pi,C_0)
	=
	O_{\mathbb P}(n^{-1/2}),\)	
	every deterministic sequence \(
	\varepsilon_n\downarrow0, \
	\sqrt n\,\varepsilon_n\to\infty,
	\)	
	satisfies the required condition
	\(d_{\mathcal C}(\widehat C_n^\Pi,C_0)
	=o_{\mathbb P}(\varepsilon_n)\).
	Such a threshold removes the \(n^{-1/2}\)-scale refinement increments
	while eventually retaining every genuine positive increment. Thus the
	unthresholded estimator recovers a binary truth with probability tending
	to one, whereas at a multifurcating truth it selects a binary refinement
	with probability tending to one and exact unthresholded recovery has
	probability tending to zero; thresholding restores consistent recovery of
	the reduced hierarchy. This scale separation parallels penalized HAC
	structure estimation: \citet{OkhrinRistig2024} impose, among other
	stage-dependent conditions, \(\sqrt n\,\lambda_n\to\infty\) to fuse
	redundant hierarchical levels, while here the same separation arises
	directly from the tangent-fan geometry.
	
	At a binary hierarchy there are no proper refinements, and hence \(
	\mathcal T_{\bar\sigma}=V_{\bar\sigma}.
	\)	
	The limiting projection is therefore the linear orthogonal projection onto
	\(V_{\bar\sigma}\), so the Gaussian cophenetic perturbation remains
	Gaussian. Moreover, uniqueness of projection is automatic in this case;
	consequently, Assumption~\assref{ass:density} is not needed for the
	following corollary, while Assumption~\assref{ass:projection} remains in
	force.
	
	Let \(P_{\bar\sigma}\) denote the orthogonal projection of \(\mathbb R^N\)
	onto \(V_{\bar\sigma}\). Regard \(L_{\bar\sigma}\) as its
	\(N\times K(\bar\sigma)\) matrix in the coordinates \((r,\eta)\), and
	write \(
	L_{\bar\sigma}^\dagger
	:=
	\left(
	L_{\bar\sigma}^{\mathsf T}L_{\bar\sigma}
	\right)^{-1}
	L_{\bar\sigma}^{\mathsf T}.
	\)
	
	On the event
	\(\widehat H_n^\Pi\in\mathfrak U^{\bar\sigma}\), write \(
	\widehat H_n^\Pi
	=
	L_{\bar\sigma}(\widehat r_n,\widehat\eta_n),
	\ 
	H_0=L_{\bar\sigma}(r_0,\eta_0),\)	
	and define these coordinates arbitrarily outside this event.

	\begin{corollary}
		\label{cor:binary-hierarchy-clt}
		Suppose Assumption \assref{ass:projection} holds. If \(\bar\sigma\in\Bd\), then
		\begin{equation}
			\label{eq:binary-hierarchy-clt}
			\sqrt n\left(
			\widehat H_n^\Pi-H_0
			\right)
			\Rightarrow
			P_{\bar\sigma}D_T^{-1}Z.
		\end{equation}
		In particular, the limit is centered Gaussian, with covariance matrix
		\(
		P_{\bar\sigma}
		D_T^{-1}\Sigma_TD_T^{-1}
		P_{\bar\sigma}.
		\)
		Moreover, 
		\begin{equation}
			\label{eq:binary-coordinate-clt}
			\sqrt n
			\begin{pmatrix}
				\widehat r_n-r_0\\
				\widehat\eta_n-\eta_0
			\end{pmatrix}
			\Rightarrow
			L_{\bar\sigma}^\dagger D_T^{-1}Z.
		\end{equation}
	\end{corollary}
	
	\begin{proof}
		A binary hierarchy has no proper refinement. Hence
		\(
		\mathcal T_{\bar\sigma}
		=
		\mathcal T_{\bar\sigma\mid\bar\sigma}
		=
		V_{\bar\sigma},
		\)
		and metric projection onto the tangent fan is the linear orthogonal
		projection \(P_{\bar\sigma}\). The proof of Theorem~\ref{thm:whole-space-singular-clt}, with this linear
		projection in place of the general tangent-fan projection, therefore gives
		\eqref{eq:binary-hierarchy-clt} without any uniqueness assumption.
		
		By Theorem~\ref{thm:thresholded-hierarchy-recovery},
		\(
		\mathbb P(
		\widehat H_n^\Pi\in\mathfrak U^{\bar\sigma}
		)\to1
		\).
		On this event,
		\[
		\begin{pmatrix}
			\widehat r_n-r_0\\
			\widehat\eta_n-\eta_0
		\end{pmatrix}
		=
		L_{\bar\sigma}^\dagger
		(\widehat H_n^\Pi-H_0).
		\]
		Since \(L_{\bar\sigma}^\dagger P_{\bar\sigma}=L_{\bar\sigma}^\dagger\),
		\eqref{eq:binary-coordinate-clt} follows.
	\end{proof}

	\section{Verification for the empirical Kendall matrix}
	\label{sec:kendall}
	
	We conclude by verifying the statistical assumptions of
	Sections~\ref{sec:statistical-recovery} and~\ref{sec:singular-clt}
	for the empirical matrix of pairwise Kendall coefficients. For this purpose we use frailties.
	
	Throughout this section, \(\mathcal K\) is one of the five
	families listed in Proposition~\ref{lem:admissible-families} and \(\psi_\theta\) is its corresponding generator.
	For every internal edge \(v\to w\), the sufficient nesting condition
	implies that
	\(
	\phi_{vw}
	:=
	\psi_{\theta_v}^{-1}\circ\psi_{\theta_w}
	\)
	is a Bernstein function. Hence, for every \(x>0\),
	\(t\mapsto\exp\{-x\phi_{vw}(t)\}\) is the Laplace transform of a
	probability distribution on \([0,\infty)\). Starting from a root
	frailty \(V_\rho\) with
	\(
	\mathbb E[\exp(-tV_\rho)]
	=
	\psi_{\theta_\rho}(t),
	\)
	we may therefore construct, recursively down the finite hierarchy,
	conditionally independent child frailties satisfying
	\[
	\mathbb E[e^{-tV_w}\mid V_v=x]
	=
	\exp\{-x\phi_{vw}(t)\}.
	\]
	We first note that all frailties are strictly positive almost surely.
	Since
	\(\psi_{\theta_\rho}(t)\to0\) as \(t\to\infty\), we have 
	\[
	\mathbb P(V_\rho=0)
	=
	\lim_{t\to\infty}\mathbb E[e^{-tV_\rho}]
	=
	0.
	\]
	Now let \(w\) be an internal child of \(v\).  Since
	\(
	\phi_{vw}(t)
	=
	\psi_{\theta_v}^{-1}
	(\psi_{\theta_w}(t))
	\to\infty\) as \( t\to\infty,
	\)
	we have, for every \(x>0\),
	\[
	\begin{aligned}
		\mathbb P(V_w=0\mid V_v=x)
		&=
		\lim_{t\to\infty}
		\mathbb E[e^{-tV_w}\mid V_v=x]=
		\lim_{t\to\infty}
		\exp\{-x\phi_{vw}(t)\}
		=
		0.
	\end{aligned}
	\]
	Induction from the root over the finite hierarchy therefore gives
	\(
	\mathbb P(V_v>0)=1\) for all \( v\in\operatorname{Int}(\sigma).
	\)
	
	This allows us to define, for each leaf \(i\), independently
	conditional on the frailties,
	\begin{equation}
		\label{eq:leaf-frailty}
	U_i
	=
	\psi_{\theta_{\parv (i)}}\!
	\left(\frac{E_i}{V_{\parv (i)}}\right),
	\qquad
	E_i\stackrel{\mathrm{iid}}{\sim}\operatorname{Exp}(1),
	\end{equation}
	This
	is the standard hierarchical frailty representation of the nested
	Archimedean copula; see \cite{McNeil2008}.
	
%
	\begin{lemma}
		\label{lem:hac-positive-density}
		Let \(\mathcal K\) be one of the five families in
		Proposition~\ref{lem:admissible-families}. Then every
		\(C\in\mathcal C_d^{\mathcal K}\) admits a density \(c\) such that
		\[
		c\in C\bigl((0,1)^d\bigr),
		\qquad
		c(u)>0,\quad u\in(0,1)^d.
		\]
	\end{lemma}
	
	\begin{proof}
		For a leaf \(i\), formula \eqref{eq:leaf-frailty} gives \(P(U_i\le u \mid V_{\parv(i)}=v)=\exp(-v\psi^{-1}_\theta(u))\). Differentiating, we get the density 
		\[
		f_i(u\mid v)
		=
		v\,
		\bigl[-(\psi_{\theta_{\parv(i)}}^{-1})'(u)\bigr]
		\exp\!\left\{
		-v\psi_{\theta_{\parv(i)}}^{-1}(u)
		\right\},
		\qquad u\in(0,1).
		\]
		For each admissible family \(K\), we have \(\psi_\theta'(t)<0\) for \(t>0\). Hence
		\(
		-(\psi_\theta^{-1})'(u)>0,
		\)
		and \(f_i(u|v)>0\). Since the leaves are independent conditionally on the frailties, an application of Tonelli's theorem gives
		\[
		c(u)
		=
		\mathbb E\!\left[
		\prod_{i=1}^d f_i(u_i\mid V_{\parv(i)})
		\right]
		>0,
		\qquad u\in(0,1)^d.
		\]
		
		Fix a compact set \(K\Subset(0,1)^d\), and choose
		\(0<\delta<1/2\) such that
		\(K\subset[\delta,1-\delta]^d\).
		For every leaf \(i\), continuity of the inverse generator and
		its derivative gives constants \(A_i<\infty\) and \(a_i>0\) such that
		\(
		-(\psi_{\theta_{\parv(i)}}^{-1})'(u)\le A_i\) and
		\(
		\psi_{\theta_{\parv(i)}}^{-1}(u)\ge a_i\) for \( u\in[\delta,1-\delta].
		\)
		For each internal vertex with
		\(\mathcal I_v:=\{i:\parv(i)=v\}\ne\varnothing\), set
		\(
		m_v:=|I_v|,\) \( M_v:=\prod_{i\in I_v}A_i\) and \(a_v:=\sum_{i\in I_v}a_i.
		\)
		Then, uniformly for \(u\in K\),
		\[
		\prod_{i\in I_v}f_i(u_i\mid V_v)
		\le M_vV_v^{m_v}e^{-a_vV_v}
		\le M_v\left(\frac{m_v}{a_ve}\right)^{m_v}.
		\]
		Multiplying over the finitely many vertices with
		\(\mathcal I_v\ne\varnothing\) gives a deterministic finite bound for
		the conditional joint density on \(K\).
		Dominated convergence therefore proves continuity of \(c\) on \(K\). As \(K\) was arbitrary, \(c\in C((0,1)^d)\).
	\end{proof}

	For \(a=1,\ldots,n\), write
	\(
	X_a=(X_{a1},\ldots,X_{ad}).
	\)
	For every \(1\le i<j\le d\), define the empirical Kendall coefficients
	by
	\[
	\widehat\tau_{ij,n}
	:=
	\binom{n}{2}^{-1}
	\sum_{1\le a<b\le n}
	\kappa_{ij}\left(X_a,X_b\right), \quad \kappa_{ij}(x,y)
	:=
	\operatorname{sign}\!\left(
	(x_i-y_i)(x_j-y_j)
	\right),
	\quad x ,y\in\R^d.
	\]
	Set
	\(
	\widehat{\mathsf T}_{\tau,n}
	:=
	(\widehat\tau_{ij,n})_{1\le i<j\le d}
	\in\R^N.
	\)
	Since the margins are continuous, ties occur with probability zero and
	\(\widehat{\mathsf T}_{\tau,n}\) depends only on the componentwise
	ranks.
	
	Define the vector-valued kernel
	\(
	\kappa(x,y)
	:=
	(
	\kappa_{ij}(x,y))_{1\le i<j\le d}.
	\)
	Its expectation is
	\(
	\mathbb E[\kappa(X_1,X_2)]
	=
	\mathsf T_\tau(C_0)
	:=
	(
	\tau(C_{0,ij})
	)_{1\le i<j\le d}.
	\)
	Finally define
	\[
	g(x)
	:=
	\mathbb E[\kappa(x,X_2)]
	-
	\mathsf T_\tau(C_0),
	\qquad
	\Sigma_\tau
	:=
	4\operatorname{Cov}(g(X_1)).
	\]

	\begin{proposition}
		\label{prop:empirical-kendall}
		Assume that \(X_1,\ldots,X_n \in \R^d\) are i.i.d. with continuous marginal distributions \((F_1, \dots, F_d)\)
		and copula \(C_0\in\mathcal C_d^{\mathcal K}\) for an admissible family \(\mathcal K\) of Proposition~\ref{lem:admissible-families}. 
		Then
		\[
		\sqrt n\left(
		\widehat{\mathsf T}_{\tau,n}
		-
		\mathsf T_\tau(C_0)
		\right)
		\Rightarrow
		Z_\tau,
		\qquad
		Z_\tau\sim\mathcal N_N(0,\Sigma_\tau).
		\]
		Moreover, \(\Sigma_\tau\) is positive definite and
		the empirical Kendall matrix satisfies
		Assumptions~\assref{ass:projection}
		and~\assref{ass:density}.
	\end{proposition}
	
	\begin{proof}
		Set
		\(
		U_a=(F_1(X_{a1}),\ldots,F_d(X_{ad})).
		\)
		Since the margins are continuous, the pairwise sample Kendall
		coefficients computed from \(X_1,\ldots,X_n\) coincide almost
		surely with those computed from \(U_1,\ldots,U_n\).
		Moreover, \(U_a\) has copula \(C_0\), whose density is continuous
		and strictly positive on \((0,1)^d\) by Lemma \ref{lem:hac-positive-density}.
		
		The joint central limit theorem for an arbitrary finite vector of
		sample Kendall coefficients follows from
		\cite[Eq.~(3.8)]{Rublik2016}; see also
		\cite[Theorem~2.3]{ElMaacheLepage2003}.
		The corresponding asymptotic covariance matrix is
		\(4\operatorname{Cov}(g(X_1))=\Sigma_\tau\).
		By \cite[Theorem~3.1(ii)]{Rublik2016}, this matrix is positive
		definite.
		
		The profile condition in Assumption~\assref{ass:projection} follows from the
		\eqref{eq:universal-kendall-profile}, and
		Assumption~\assref{ass:density} follows from the positive definiteness of
		\(\Sigma_\tau\).
	\end{proof}

Consequently, for the empirical Kendall matrix all the conclusions of Theorems~\ref{thm:projected-estimator-consistency}, \ref{thm:thresholded-hierarchy-recovery} and \ref{thm:whole-space-singular-clt} apply, with Corollary~\ref{cor:binary-hierarchy-clt} providing the corresponding Gaussian specialization at binary hierarchies.

\section{Discussion}
\label{sec:discussion}

The main purpose of this paper has been to place hierarchy and continuous
dependence parameters in a single reduced model space.  After identifying
decorated hierarchies that represent the same copula, homogeneous
hierarchical Archimedean copulas acquire the polyhedral geometry of the
ultrametric fan: reduced hierarchies index its strata, contraction of an
internal edge corresponds to passage to a face, and all admissible
refinements of a multifurcating hierarchy meet at the same lower-dimensional
model.  In this representation, identifiability, hierarchy recovery and
singular asymptotics are different manifestations of the same local
geometry.

The pairwise-coordinate construction separates this geometry from the
particular dependence summary used to observe it.  By
Theorem~\ref{thm:informative-functionals}, an informative bivariate
functional acts componentwise on cophenetic coordinates through a scalar
profile.  Once this profile is inverted, Kendall's tau, Spearman's rho and
the informative tail-dependence coefficients considered above all recover
the same underlying reduced hierarchy space.  Pairwise dependence summaries
are therefore not only devices for constructing a hierarchy: under the
stated identifiability conditions, their complete matrices provide global
coordinates on the reduced model class.  This is also why the projection
estimator can be defined without first selecting a binary hierarchy.

The asymptotic results reveal a corresponding distinction between
estimation of the copula and recovery of its exact combinatorial structure.
At a binary truth, the parameter lies in the relative interior of a maximal
cone.  The tangent fan is then a linear space, the projected Gaussian
perturbation remains Gaussian, and the hierarchy is locally stable.  At a
multifurcating truth, several refinement cones meet.  The estimator remains
root-\(n\) consistent, but Theorem~\ref{thm:whole-space-singular-clt} shows
that its first-order fluctuation is obtained by projection onto the full
refinement fan.  The singularity is therefore expressed not through a
slower estimation rate, but through random first-order selection among the
refinement strata meeting at the true model.  The additional refinement
edges have size \(n^{-1/2}\), and
Theorem~\ref{thm:thresholded-hierarchy-recovery} shows that contracting
increments at any asymptotically larger vanishing scale separates these
sampling-induced resolutions from the genuine positive edges of the
reduced hierarchy.

The Euclidean cophenetic metric used here is deliberately structural rather
than information-theoretic.  A first complementary direction is therefore
to study the statistical geometry induced by the likelihood itself.  On
each regular hierarchy stratum, Fisher information defines a natural
Riemannian metric, but the relevant question for the reduced model is what
happens when an internal edge contracts and several parametrizations meet.
One would like to determine whether the stratumwise Fisher tensors admit
compatible limits on common contraction faces, whether likelihood
differentiability persists in hierarchy-opening directions, and how the
resulting information geometry compares with the cophenetic geometry.
There is no reason a priori for this behavior to be family-independent:
the same polyhedral contraction may correspond to very different local
likelihood separation in different Archimedean families.  Thus the
polyhedral fan describes which structural alternatives meet, whereas the
Fisher geometry should quantify how statistically distinguishable those
alternatives are.  Such an analysis would also clarify the relation between
the projection asymptotics developed here and likelihood-based constrained
inference at structural singularities.

A second complementary direction concerns increasing structural complexity.
The present analysis keeps \(d\) fixed, so the number of hierarchy strata
and refinement directions is finite and the local sampling scale is
\(n^{-1/2}\).  When \(d=d_n\) grows, however, two additional effects appear.
The number of competing structural alternatives increases, and the smallest
internal increment of a hierarchy may itself approach zero.  Exact
hierarchy recovery should then depend not only on the local estimation
error but also on the separation of the true parameter from the
non-binary skeleton and on the multiplicity of competing refinements.
For estimators built from pairwise summaries, simultaneous control of the
\(\binom{d}{2}\) coordinates naturally introduces a multiple-comparison
effect beyond the fixed-dimensional root-\(n\) scale.  More generally, one
is led to ask for the smallest internal edge that can be statistically
resolved as a function of sample size, dimension and the local information
carried by the chosen copula family.  Below the scale of exact recovery,
thresholded recovery suggests a weaker but still meaningful target: recover
the hierarchy only up to structural features that are statistically
distinguishable at the available resolution.

These two directions point to a broader interpretation of the present
results.  Structural recovery near a singular model is governed by three
different ingredients.  The local stratified geometry determines which
alternative structures meet and in which directions they can separate;
the statistical experiment determines the information available in those
directions; and the number of competing structures determines how local
distinguishability translates into global recovery.  In the fixed-dimensional
setting studied here, the first ingredient already explains the tangent-fan
limit and the need for thresholding.  Incorporating information geometry
and growing structural complexity would turn this geometric picture into
a theory of statistical resolution: one would ask at what scale its
structural features can actually be distinguished from data.

\appendix
\section{Concavity calculations}
\label{app:concavity}
Fix \(\theta_1<\theta_2\) in the parameter interval of one of the
five families and define 
\(
F_{12}^{\cl K}(t)
:=
\psi_{\theta_1}^{-1}\bigl(\psi_{\theta_2}(t)\bigr),
\ t\ge0,
\) for all admissible families \(\mathcal K\).
If \(F_{12}^{\cl K}\) is concave, then, since \(F_{12}^{\cl K}(0)=0\),
concavity gives
\(
F_{12}^{\cl K}(x+y)
\le
F_{12}^{\cl K}(x)+F_{12}^{\cl K}(y).
\)
Hence, writing \(u=\psi_{\theta_2}(x)\) and
\(v=\psi_{\theta_2}(y)\), and using that
\(\psi_{\theta_1}\) is decreasing,
\[
\begin{aligned}
	C^{(2)}_{\theta_1}(u,v)
	&=
	\psi_{\theta_1}
	\bigl(F_{12}^{\cl K}(x)+F_{12}^{\cl K}(y)\bigr)\le
	\psi_{\theta_1}\bigl(F_{12}^{\cl K}(x+y)\bigr)
	=
	\psi_{\theta_2}(x+y)
	=
	C^{(2)}_{\theta_2}(u,v).
\end{aligned}
\]
It therefore remains only to check concavity of \(F_{12}^{\cl K}\).

Let \(a:=\theta_1/\theta_2\in(0,1)\). For Clayton and Gumbel,
\(
F_{12}^{\operatorname{Cl}}(t)=(1+t)^a-1
\) and \(
F_{12}^{\operatorname{Gu}}(t)=t^a
\), which are concave. 	
For AMH, direct substitution gives
\[
F_{12}^{\operatorname{AMH}}(t)
=
\log
\frac{(1-\theta_1)e^t+\theta_1-\theta_2}
{1-\theta_2}.
\]
Since
\[
(F_{12}^{\operatorname{AMH}})''(t)
=
\frac{(1-\theta_1)e^t(\theta_1-\theta_2)}
{\bigl((1-\theta_1)e^t+\theta_1-\theta_2\bigr)^2}
<0,
\]
\(F_{12}^{\operatorname{AMH}}\) is concave.

For Joe and Frank, define
\(
G_{a,c}(t)
:=
-\log\!\left(
1-(1-ce^{-t})^a
\right),\) with \(
\
0<a\le1,\ 0<c\le1.
\)
Writing \(s=1-ce^{-t}\), one obtains
\(
G_{a,c}'(t)
=
a(1-s)s^{a-1}/(1-s^a)=:H(s),
\) and
\[
H'(s)
=
-\frac{
	a s^{a-2}\bigl(1-a+as-s^a\bigr)
}{
	(1-s^a)^2
}
\le0,
\]
because concavity of \(s\mapsto s^a\) gives
\(s^a\le 1-a+as\). Thus \(G_{a,c}\) is concave. Consequently, for Joe,
\(
F_{12}^{\operatorname{Jo}}(t)
=
-\log\!\left(
1-(1-e^{-t})^a
\right)
\)
is concave.
Finally, for Frank, put
\(
c:=1-e^{-\theta_2}.
\)
Using
\[
\psi_{\theta}^{-1}(u)
=
-\log
\frac{1-e^{-\theta u}}{1-e^{-\theta}},
\]
we obtain
\(
F_{12}^{\operatorname{Fr}}(t)
=
\log(1-e^{-\theta_1})
-
\log(
1-(1-ce^{-t})^a
),
\)
which is concave. Thus all five families are positively
ordered:
\[
\theta_1<\theta_2
\quad\Longrightarrow\quad
C^{(2)}_{\theta_1}(u,v)
\le
C^{(2)}_{\theta_2}(u,v),
\qquad (u,v)\in[0,1]^2.
\]
	\section*{Acknowledgements}
	This research did not receive any specific grant from funding agencies in the public, commercial, or not-for-profit sectors.

	\printcredits
	
	\bibliographystyle{cas-model2-names}
	
	\bibliography{biblio_copula}

@article{Ardila2005,
  author  = {Ardila, Federico},
  title   = {Subdominant matroid ultrametrics},
  journal = {Annals of Combinatorics},
  year    = {2005},
  volume  = {8},
  pages   = {379--389},
  doi     = {10.1007/s00026-004-0227-1}
}

@article{ArdilaKlivans2006,
  author  = {Ardila, Federico and Klivans, Caroline J.},
  title   = {The {Bergman} complex of a matroid and phylogenetic trees},
  journal = {Journal of Combinatorial Theory, Series B},
  year    = {2006},
  volume  = {96},
  pages   = {38--49},
  doi     = {10.1016/j.jctb.2005.06.004}
}

@article{BardenLeOwen2013,
  author  = {Barden, Dennis and Le, Huiling and Owen, Megan},
  title   = {Central limit theorems for {Fr\'echet} means in the space of phylogenetic trees},
  journal = {Electronic Journal of Probability},
  year    = {2013},
  volume  = {18},
  pages   = {1--25},
  doi     = {10.1214/EJP.v18-2201}
}

@article{CardonaMirRosselloRotgerSanchez2013,
  author  = {Cardona, Gabriel and Mir, Arnau and Rossell{\'o}, Francesc and Rotger, Lloren\c{c} and S{\'a}nchez, David},
  title   = {Cophenetic metrics for phylogenetic trees, after {Sokal} and {Rohlf}},
  journal = {BMC Bioinformatics},
  year    = {2013},
  volume  = {14},
  pages   = {3},
  doi     = {10.1186/1471-2105-14-3}
}

@article{Drton2009,
  author  = {Drton, Mathias},
  title   = {Likelihood ratio tests and singularities},
  journal = {The Annals of Statistics},
  year    = {2009},
  volume  = {37},
  pages   = {979--1012},
  doi     = {10.1214/07-AOS571}
}

@misc{FaheyHostenKriegerTimpe2008,
  author       = {Fahey, Charles and Ho{\c{s}}ten, Serkan and Krieger, Nathan and Timpe, Leslie},
  title        = {Least squares methods for equidistant tree reconstruction},
  year         = {2008},
  eprint       = {0808.3979},
  archivePrefix= {arXiv},
  note         = {Unpublished preprint}
}

@article{Geyer1994,
  author  = {Geyer, Charles J.},
  title   = {On the asymptotics of constrained $M$-estimation},
  journal = {The Annals of Statistics},
  year    = {1994},
  volume  = {22},
  number  = {4},
  pages   = {1993--2010},
  doi     = {10.1214/aos/1176325768}
}

@article{GoreckiHofertHolena2016,
  author  = {G{\'o}recki, Jan and Hofert, Marius and Hole{\v{n}}a, Martin},
  title   = {An approach to structure determination and estimation of hierarchical {Archimedean} copulas and its application to {Bayesian} classification},
  journal = {Journal of Intelligent Information Systems},
  year    = {2016},
  volume  = {46},
  pages   = {21--59},
  doi     = {10.1007/s10844-014-0350-3}
}

@article{GoreckiHofertHolena2017,
  author  = {G{\'o}recki, Jan and Hofert, Marius and Hole{\v{n}}a, Martin},
  title   = {On structure, family and parameter estimation of hierarchical {Archimedean} copulas},
  journal = {Journal of Statistical Computation and Simulation},
  year    = {2017},
  volume  = {87},
  number  = {17},
  pages   = {3261--3324},
  doi     = {10.1080/00949655.2017.1365148}
}

@article{HofertMaechler2011,
  author  = {Hofert, Marius and M{\"a}chler, Martin},
  title   = {Nested {Archimedean} copulas meet {R}: The nacopula package},
  journal = {Journal of Statistical Software},
  year    = {2011},
  volume  = {39},
  pages   = {1--20},
  doi     = {10.18637/jss.v039.i09}
}

@article{HofertMaechlerMcNeil2012,
  author  = {Hofert, Marius and M{\"a}chler, Martin and McNeil, Alexander J.},
  title   = {Likelihood inference for {Archimedean} copulas in high dimensions under known margins},
  journal = {Journal of Multivariate Analysis},
  year    = {2012},
  volume  = {110},
  pages   = {133--150},
  doi     = {10.1016/j.jmva.2012.02.019}
}

@article{HofertPham2013,
  author  = {Hofert, Marius and Pham, David},
  title   = {Densities of nested {Archimedean} copulas},
  journal = {Journal of Multivariate Analysis},
  year    = {2013},
  volume  = {118},
  pages   = {37--52},
  doi     = {10.1016/j.jmva.2013.03.006}
}

@article{HotzEtAl2013,
  author  = {Hotz, Thomas and Huckemann, Stephan and Le, Huiling and Marron, J. S. and Mattingly, Jonathan C. and Miller, Ezra and Nolen, James and Owen, Megan and Patrangenaru, Vic and Skwerer, Sean},
  title   = {Sticky central limit theorems on open books},
  journal = {The Annals of Applied Probability},
  year    = {2013},
  volume  = {23},
  pages   = {2238--2258},
  doi     = {10.1214/12-AAP899}
}

@book{Joe1997,
  author    = {Joe, Harry},
  title     = {Multivariate Models and Multivariate Dependence Concepts},
  series    = {Monographs on Statistics and Applied Probability},
  volume    = {73},
  publisher = {Chapman \& Hall},
  address   = {London},
  year      = {1997}
}

@article{MatsypuraNeoProkhorov2016,
  author  = {Matsypura, Dmytro and Neo, E. and Prokhorov, Artem},
  title   = {Estimation of hierarchical {Archimedean} copulas as a shortest path problem},
  journal = {Economics Letters},
  year    = {2016},
  volume  = {149},
  pages   = {131--134},
  doi     = {10.1016/j.econlet.2016.10.034}
}

@article{McNeil2008,
  author  = {McNeil, Alexander J.},
  title   = {Sampling nested {Archimedean} copulas},
  journal = {Journal of Statistical Computation and Simulation},
  year    = {2008},
  volume  = {78},
  number  = {6},
  pages   = {567--581},
  doi     = {10.1080/00949650701255834}
}

@article{McNeilNeslehova2009,
  author  = {McNeil, Alexander J. and Ne{\v{s}}lehov{\'a}, Johanna},
  title   = {Multivariate {Archimedean} copulas, $d$-monotone functions and $\ell_1$-norm symmetric distributions},
  journal = {The Annals of Statistics},
  year    = {2009},
  volume  = {37},
  number  = {5B},
  pages   = {3059--3097},
  doi     = {10.1214/07-AOS556}
}

@book{Nelsen2006,
  author    = {Nelsen, Roger B.},
  title     = {An Introduction to Copulas},
  edition   = {2},
  publisher = {Springer},
  address   = {New York},
  year      = {2006}
}

@article{OkhrinOkhrinSchmid2013,
  author  = {Okhrin, Ostap and Okhrin, Yarema and Schmid, Wolfgang},
  title   = {On the structure and estimation of hierarchical {Archimedean} copulas},
  journal = {Journal of Econometrics},
  year    = {2013},
  volume  = {173},
  number  = {2},
  pages   = {189--204},
  doi     = {10.1016/j.jeconom.2012.12.001}
}

@article{OkhrinOkhrinSchmidProperties2013,
  author  = {Okhrin, Ostap and Okhrin, Yarema and Schmid, Wolfgang},
  title   = {Properties of hierarchical {Archimedean} copulas},
  journal = {Statistics \& Risk Modeling},
  year    = {2013},
  volume  = {30},
  pages   = {21--54},
  doi     = {10.1524/strm.2013.1071}
}

@article{OkhrinRistig2024,
  author  = {Okhrin, Ostap and Ristig, Alexander},
  title   = {Penalized estimation of hierarchical {Archimedean} copula},
  journal = {Journal of Multivariate Analysis},
  year    = {2024},
  volume  = {201},
  pages   = {105274},
  doi     = {10.1016/j.jmva.2023.105274}
}

@article{PerreaultTangPanReid2025,
  author  = {Perreault, Samuel and Tang, Yanbo and Pan, Ruyi and Reid, Nancy},
  title   = {Inference for overparametrized hierarchical {Archimedean} copulas},
  journal = {Journal of Multivariate Analysis},
  year    = {2025},
  volume  = {210},
  pages   = {105483},
  doi     = {10.1016/j.jmva.2025.105483}
}

@article{SegersUyttendaele2014,
  author  = {Segers, Johan and Uyttendaele, Nathan},
  title   = {Nonparametric estimation of the tree structure of a nested {Archimedean} copula},
  journal = {Computational Statistics \& Data Analysis},
  year    = {2014},
  volume  = {72},
  pages   = {190--204},
  doi     = {10.1016/j.csda.2013.10.028}
}

@book{SempleSteel2003,
  author    = {Semple, Charles and Steel, Mike},
  title     = {Phylogenetics},
  series    = {Oxford Lecture Series in Mathematics and Its Applications},
  volume    = {24},
  publisher = {Oxford University Press},
  address   = {Oxford},
  year      = {2003}
}

@article{Shapiro2000,
  author  = {Shapiro, Alexander},
  title   = {On the asymptotics of constrained local $M$-estimators},
  journal = {The Annals of Statistics},
  year    = {2000},
  volume  = {28},
  pages   = {948--960},
  doi     = {10.1214/aos/1015952006}
}

@misc{Uyttendaele2014,
  author        = {Uyttendaele, Nathan},
  title         = {Nested {Archimedean} copulas: A new class of nonparametric tree structure estimators},
  year          = {2014},
  eprint        = {1407.0230},
  archivePrefix = {arXiv}
}

@article{Uyttendaele2018,
  author  = {Uyttendaele, Nathan},
  title   = {On the estimation of nested {Archimedean} copulas: A theoretical and an experimental comparison},
  journal = {Computational Statistics},
  year    = {2018},
  volume  = {33},
  pages   = {1047--1070},
  doi     = {10.1007/s00180-017-0743-1}
}

@article{Rublik2016,
	author  = {Rubl{\'i}k, Franti{\v{s}}ek},
	title   = {Estimates of the Covariance Matrix of Vectors of
	{U}-Statistics and Confidence Regions for Vectors of
	Kendall's Tau},
	journal = {Kybernetika},
	year    = {2016},
	volume  = {52},
	number  = {2},
	pages   = {280--293},
	doi     = {10.14736/kyb-2016-2-0280}
}

@incollection{ElMaacheLepage2003,
	author    = {El Maache, Hamani and Lepage, Yves},
	title     = {Spearman's Rho and Kendall's Tau for Multivariate Data Sets},
	booktitle = {Mathematical Statistics and Applications:
	Festschrift for Constance van Eeden},
	series    = {Lecture Notes--Monograph Series},
	volume    = {42},
	pages     = {113--130},
	year      = {2003},
	publisher = {Institute of Mathematical Statistics},
	address   = {Beachwood, OH},
	doi       = {10.2307/4356234}
}
	
	

\end{document}